\documentclass[a4paper]{amsart}
\usepackage{graphicx}
\usepackage{amsmath}
\usepackage{latexsym}
\usepackage{amssymb}
\usepackage{subfigure}
\usepackage{epsf}
\usepackage[cmtip,matrix,arrow,curve]{xy}
\UseComputerModernTips

\DeclareMathOperator{\doog}{{\mathcal F}(\Delta)}
\DeclareMathOperator{\uHom}{\underline{{Hom}}}

\newcommand{\eps}{\epsilon}
\newcommand{\keps}{\epsilon_{\ast}}
\newcommand{\Hom}{\operatorname{Hom}\nolimits}

\newcommand{\End}{\operatorname{End}\nolimits}
\renewcommand{\Im}{\operatorname{Im}\nolimits}

\newcommand{\Ext}{\operatorname{Ext}\nolimits}
\newcommand{\ext}{\operatorname{ext}\nolimits}

\newcommand{\soc}{\operatorname{soc}\nolimits}

\newcommand{\SO}{\mathrm{SO}}
\newcommand{\Sp}{\mathrm{Sp}}
\newcommand{\SL}{\mathrm{SL}}

\newcommand{\Lie}{\operatorname{Lie}\nolimits}

\begin{document}

\newtheorem{lemma}{Lemma}[section]
\newtheorem{prop}[lemma]{Proposition}
\newtheorem{cor}[lemma]{Corollary}
\newtheorem{theorem}[lemma]{Theorem}
\newtheorem{remark}[lemma]{Remark}
\newtheorem{definition}[lemma]{Definition}
\newtheorem{conj}[lemma]{Conjecture}

\theoremstyle{definition}
\newtheorem{ex}[lemma]{Example}

\title[$\Delta$-filtered modules and $P$-orbits]{$\Delta$-filtered modules and nilpotent 
orbits of a parabolic subgroup in O$_N$}

\author[Baur]{Karin Baur}
\address{Department of Mathematics \\
ETH Z\"urich
R\"amistrasse 101 \\
8092 Z\"urich \\
Switzerland
}
\email{baur@math.ethz.ch}

\author[Erdmann]{Karin Erdmann}
\address{Mathematical Institute \\
24-29 St. Giles \\
Oxford OX1 3LB \\
UK
}
\email{erdmann@maths.ox.ac.uk}

\author[Parker]{Alison Parker}
\address{Department of Pure Mathematics \\
University of Leeds\\
Leeds, LS2 9JT \\
UK
}
\email{parker@maths.leeds.ac.uk}


\begin{abstract}
We study certain $\Delta$-filtered modules for the Auslander
algebra of $k[T]/T^n\rtimes C_2$ where $C_2$ is the cyclic group 
of order two. 
The motivation of this lies in the problem of describing the $P$-orbit 
structure for the action of a parabolic subgroup $P$ of an orthogonal group.
For any parabolic subgroup of an 
orthogonal group we construct a map
from parabolic orbits to $\Delta$-filtered modules and show that in
the case of the Richardson orbit, the resulting module has no self-extensions. 
\end{abstract} 

\maketitle

%
\section*{Introduction}
%



Let $k$ be an algebraically
closed field of characteristic different from $2$, 
and 
let $G$ be a reductive algebraic group over $k$, $P\subset G$ a parabolic subgroup. 
Now $P$ acts on its unipotent radical $U$ by conjugation and on the nilradical 
$\mathfrak{n}=\Lie U$ by the adjoint action. 
By a fundamental result of Richardson (\cite{ri}), this action has an open dense orbit, the 
so-called {\em Richardson orbit} of $P$. But in general, the number of orbits is 
not finite and it is a very hard problem to understand the orbit structure. 
The question of deciding whether $P$ has a finite 
number of orbits in $\mathfrak{n}$ has been asked by Popov and R\"ohrle in~\cite{pv}. 
For groups in characteristic zero or in good characteristic,
the parabolic subgroups 
with finitely many orbits on their nilradical have been classified in a sequence of papers,
~\cite{ro},~\cite{hr97} and~\cite{hr}. 

If $P$ is a parabolic subgroup of $\SL_N$, then there is an explicit  description of 
the $P$-orbits in work of Hille-R\"ohrle \cite{hr}, and Br\"ustle et
al.  \cite{bhrr}, via a connection with a quasi-hereditary algebra, 
namely the Auslander algebra $A_n$ of the truncated polynomial ring 
$R_n:= k[T]/T^n$. They have shown that the $P$-orbits are in bijection
with 
the isomorphism classes of certain $\Delta$-filtered modules of $A_n$
with no self-extensions. This list has finitely many indecpomposable modules,
parametrized as $\Delta(I)$ where $I$ runs through the subsets
of $\{ 1, 2, \ldots, n\}$.


Our main goal in this paper is to establish an analogous 
correspondence between $P$-orbits for parabolic subgroups of the special 
orthogonal groups $\SO_N$ and certain 
$\Delta$-filtered modules for the Auslander algebra of
$k[T]/T^n\rtimes C_2$, the skew group ring of one considered by
\cite{bhrr}, where $C_2$ is a cyclic group of order two. 
This article establishes the Auslander 
algebra of $k[T]/T^n\rtimes C_2$ as the correct candidate for such a correspondence. 
There is one major difference as compared with~\cite{bhrr}. In our situation,
a list of all  $\Delta$-filtered modules with no self extensions is
difficult to obtain and may even be infinite.
This must be expected however, as the construction of the
Richardson elements for $\SO_N$ involves symmetric diagrams and hence 
gives rise to symmetric $\Delta$-dimension vectors. 
Here, we use signed sets $I$, that
is certain subsets of $\{ \pm 1, \pm 2, \ldots, \pm n\}$. 
We associate to each signed set
$I$ another set $J$ (a symmetric complement) and an  extension $E(I,J)$ that is
$\Delta$-filtered with no self extensions and has the required
symmetric $\Delta$-dimension vector. These extensions may then
be used to construct a $\Delta$-filtered module that corresponds to
the Richardson element, using the work of Baur \cite{ba}
and Baur and Goodwin \cite{bg} on orthogonal Lie algebras.

We now  summarise our paper. 
The first section describes the problem of determining the parabolic orbits 
in the nilradical and introduce the notation from the side of $P$-orbits. We also 
recall the correspondence between $P$-orbits and $\Delta$-filtered modules for  
the Auslander algebra of $k[T]/T^n$ in case $P$ is a 
subgroup of $\SL_N$ as given in ~\cite{hr} and~\cite{bhrr}. 

The key ingredient of our approach is a quasi-hereditary algebra $D_n$ 
which is the Auslander algebra of 
$S_n=k[T]/T^n\rtimes C_2$.
The algebra $S_n$ and its Auslander algebra have
very similar properties as $R_n$ and the Auslander algebra
of $R_n$ (cf. Appendix, here we also show that $S_n$ is a skew group ring over
$R_n$ which is an essential tool for us.)
In Section~\ref{Dnprops} we explicitly describe the standard, costandard, 
projective and tilting modules. 

In Section~\ref{s:F(D)} we study 
$\Delta$-filtered modules for the 
Auslander algebra of $S_n$. We continue this in Section~\ref{s:D(I)} 
where we define a certain special class of $\Delta$-modules. This class will 
be essential in the construction of representatives 
of $P$-orbits. We denote these modules by $\Delta(I)$ 
where $I$ is a subset if $\{1^{\pm},\dots,n^{\pm}\}$. 
They can be seen as weighted versions of the modules 
introduced in Section 2 of~\cite{bhrr}. 

In order to construct the Richardson orbit using $D_n$-modules we need 
to understand the extension groups between the $\Delta(I)$ introduced 
in Section~\ref{s:D(I)}. 
In general, this is a very hard problem but 
we are able to determine them under certain conditions which are always 
satisfied in the setup we use. 

The Auslander algebra of $S_n$, (denoted $D_n$) is isomorphic to a skew group algebra of
the Auslander algebra of $R_n$ (denoted $A_n$). 
Every $D_n$-module is relative $A_n$-projective, and inducing and restricting 
preserves modules with $\Delta$-filtrations, as is explained in Section~\ref{indres}. 
In Proposition~\ref{prop:ext} we give an explicit result relating extension groups 
for $A_n$ to extension groups for $D_n$. 
%
Then the extensions between different modules of the form $\Delta(I)$ 
are discussed in Section~\ref{s:exts}. In particular, we present a combinatorial 
way to compute the dimension of the groups of homomorphisms 
for the algebras $A_n$ and $D_n$ using so-called initial segments of 
the subsets of $\{1,\dots,n\}$ in Proposition 8.4 and Lemma 8.6. 

Let us emphasise that in one sense, 
there are far more  $\Delta$-filtered modules with no 
self-extensions for $D_n$ than there are for algebras $A_n$. 
In Section 9 we construct such modules for $D_n$ which do not exist for $A_n$, and 
we call these
``type II modules'' (and our results produce  examples of these). 
They arise as the indecomposable extensions between two modules $\Delta(I)$ and 
$\Delta(J)$ where $J$ is dependent on $I$.
As a preparation for this, we calculate the dimensions of the $\Ext$
groups of
such pairs $\Delta(I)$, $\Delta(J)$ and show that they can grow arbitrarily 
large, cf. Lemmata \ref{lem:extAIJ} and \ref{lem:extBIJ}.
From these extensions, we obtain a module $E(I,J)$ without
self-extensions in
Proposition \ref{prop:E}. 
$E(I,J)$ is called a type I module if it is of the form $\Delta(I_1)\oplus \Delta(I_2)$ for 
some subsets $I_i$. Otherwise, $E(I,J)$ is said to be of type II. 
We study these type II modules in Section~\ref{s:class-exts}. 
In particular, we explain that type II modules occur roughly in half of the cases of 
modules obtained as extensions between such $\Delta(I)$, $\Delta(J)$. 
The characterisation of the other cases, namely of the decomposable ones 
is presented in Theorem \ref{thm:Edecomp}.

Finally, in Section~\ref{s:construction} we explain how to construct 
a module without self-extensions starting from the dimension vector 
${\bf d}$ for a parabolic subgroup of SO$_N$. 
The resulting module, $M({\bf d})$, has no self-extension and its 
restriction to $A_n$ has $\Delta$-support ${\bf d}$ (Proposition 11.6). 
The construction uses modules 
of the form $\Delta(I)$ and also the  type II modules defined before. 
We observe that the type II modules can be viewed as module theoretic 
analogues of the 
branched line diagrams appearing in the construction of Richardson elements
in~\cite{ba} and of the diagrams $D(d)$ with branched arrows 
in Section 2 of~\cite{bg}, both for the orthogonal Lie algebras: 
In the same way as these branched diagrams are needed to obtain the 
dense orbit (i.e. the Richardson orbit), the type II modules are essential 
in the construction of a module without self-extension.

%
\section{Motivation}\label{s:motivation}
%

Our ultimate goal is to find a bijection between $P$-orbits for 
a parabolic subgroup of an orthogonal group and 
isomorphism classes of certain $\Delta$-filtered modules for 
an algebra $D_n$. 
What we present in this article, is the first step towards this. Namely, 
we produce a module theoretic analogue of the 
Richardson orbit for $P$: a $D_n$-module without self-extension 
whose $\Delta$-dimension vector (after restriction to $A_n$) 
is equal to the composition ${\bf d}$ which determines 
the parabolic subgroup $P$.

Let us start by explaining the Lie algebra side of the problem. 
Let $P\subset G$ be a parabolic subgroup of a 
reductive algebraic group $G$ over 
an algebraically closed field $k$. Let $\mathfrak{p}\subset\mathfrak{g}$ be the 
corresponding Lie algebras and let 
$\mathfrak{p}=\mathfrak{l}\oplus\mathfrak{n}$ be a Levi decomposition of 
$\mathfrak{p}$, i.e. $\mathfrak{l}$ is a Levi factor of $\mathfrak{p}$ 
and $\mathfrak{n}$ is the corresponding nilradical. 
It is a very ambitious goal to understand the $P$-orbit structure 
in $\mathfrak{n}$. A first step towards this is a fundamental theorem 
of Richardson, cf.~\cite{ri}. It states that 
$P$ has an open dense orbit in $\mathfrak{n}$, the so-called 
{\em Richardson orbit} of $P$. Its elements are called 
{\em Richardson elements} for $P$. 
Observe that the existence of a dense orbit does not imply that there 
are only finitely many, 
in general, there are infinitely many $P$-orbits in the nilradical
$\mathfrak{n}$. 

For classical groups (in zero or good characteristic) 
there exists a classification 
of parabolic subgroups with a finite number of orbits due to Hille and 
R\"ohrle, cf.~\cite{hr}. 

For $G=\SL_N$, the special linear group, we can actually say more. 
In order to explain this, we need some notation. For all classical 
Lie groups we will choose the Borel subgroup $B$ to be 
the upper triangular matrices in $G$ and the maximal torus 
to be the diagonal matrices in $G$. Let $\mathfrak{b}$ and 
$\mathfrak{h}$ be the corresponding Lie algebras. 
Then $P$ is called {\em standard}, if $P\supset B$. Similarly, 
we call $\mathfrak{p}$ and $\mathfrak{l}$ {\em standard}, 
if $\mathfrak{p}\supset\mathfrak{b}$ and $\mathfrak{l}\supset\mathfrak{h}$ 
respectively. 
After suitable conjugation, we can assume that $P$, $\mathfrak{p}$ 
and $\mathfrak{l}$ are standard. 
Then $\mathfrak{l}$ consists of the matrices whose non-zero entries only 
lie in a sequence of square blocks on the diagonal. So 
$\mathfrak{l}=\mathfrak{l}({\bf d})$ where ${\bf d}=(d_1,\dots,d_n)$ is 
a composition of $N$, i.e. $\sum d_i=N$, describing the sizes of 
these square blocks. And $\mathfrak{p}=\mathfrak{p}({\bf d})$ 
consists of the matrices with zeroes below the sequence 
of square matrices of size $d_1,\dots,d_n$ on the diagonal. 
We call ${\bf d}$ a \emph{dimension vector}. 

From now on we 
will always assume that $P$, $\mathfrak{p}$ and $\mathfrak{l}$ are 
standard. 

We are ready to formulate the result of Hille and R\"ohrle, cf.~\cite{hr}: 
Let $P=P({\bf d})\subset \SL_N$. 
Then there is a bijection:
\[
\{P\mbox{-orbits in } \mathfrak{n}\}\ \stackrel{1:1}{\longleftrightarrow}\ 
\{M\in \mathcal F(\Delta)\mid \underline{\dim}_{\Delta}M={\bf d}\}/\sim
\]
where $\mathcal F(\Delta)$ are the $\Delta$-filtered modules for the 
algebra A$_n$ described in the Appendix. 
Furthermore, the Richardson orbit is mapped to the unique 
$M$ with $\Delta$-dimension vector ${\bf d}$ which has no self-extensions. 
In other words, for $P=P({\bf d})$ in $\SL_N$ there is a description 
of $P$-orbits in $\mathfrak{n}$ 
via $\Delta$-filtered modules (for $A_n$) of dimension vector ${\bf d}$. 
\\
So far, it has remained an open problem to find an analogous result 
for the other classical types. In this paper, we show that the algebra $D_n$ 
as defined in Section~\ref{Dnprops} is the right candidate to describe $P$-orbits 
in $\mathfrak{n}$ for 
$G=\SO_{N}$, the special orthogonal group. 
Thus we are providing the first part of an analogue of the correspondence above 
for classical groups: 
Using the knowledge of the algebra $D_n$ (see Appendix) 
and our results about the extensions 
between $\Delta$-filtered modules for $D_n$ (cf. Section~\ref{s:gaps} below) 
we will construct certain 
$\Delta$-filtered modules without self-extensions 
of $\Delta$-dimension vector ${\bf d}$. 
There is an interesting new phenomenon appearing in this case: 
the construction of the modules of a given dimension vector leads
to a new class of $\Delta$-filtered modules which does 
not exist for $\SL_N$. In that aspect, the situation clearly 
differs from the case of $\SL_N$. 

To be more precise: to any given dimension vector ${\bf d}$ 
we associate a $\Delta$-filtered $D_n$-module $M({\bf d})$ which 
has no self-extensions and thus this is a module theoretic counterpart 
of the dense orbit. 

The construction is given Section~\ref{s:construction} below, once 
we have all the material needed.

%
\section{The quasi-hereditary algebra $D_n$}\label{Dnprops}
%

Let $D_n$ be the Auslander algebra of $S_n$. 
It is given by the quiver $\Gamma_n$ with 
vertices $i^{\pm}$, arrows $\alpha_{i^{\pm}}$ 
and $\beta_{i^{\pm}}$ and 
by the relations in Definition~\ref{defn:Gamma_n}
(as defined in the appendix). 

Here we describe the standard, costandard, projective and tilting 
modules for $D_n$. 
We use $L(i^{\pm})$ to denote the simple $D_n$-module 
corresponding 
to the indecomposable module $M(i^{\pm})$. 
From the definition, it is straightforward to write down 
the projective 
$D_n$-modules $P(i^{\pm})= e_{i^{\pm}}D_n$ 
(where $e_{i^{\pm}}$ 
is the primitive idempotent at $i^{\pm}$). 
The indecomposable projective modules embed into each other as
follows:
\[
P(1^{\pm})\supset P(2^{\mp})\supset\dots\supset
P((2k)^{\mp}) \supset P((2k+1)^{\pm})\supset\dots
\leqno{(*)}
\]
The projective module $P(1^{\pm})$ is the injective 
hull of $L(1^{\pm})$ if $n$ is odd and of $L(1^{\mp})$ 
if $n$ is even. 

Using the indecomposable projective modules, 
we can define the standard modules 
$\Delta(i^{\pm})$ as their successive quotients, we set 
\[
\Delta(i^+)=P(i^+)/P((i+1)^-)\ \ \ \ \Delta(i^-)=P(i^-)/P((i+1)^+)
\]
So $\Delta(i^{\pm})$ has socle $1^{\pm}$. 
In particular, the $P(i^{\pm})$ are filtered by 
standard modules, and the quotients can be read off from 
(*).


For $1\leq i\leq n$, the costandard $D_n$-module $\nabla(i^+)$
is the serial module of length $i$ with socle $L(i^+)$ 
and top $L(1^*)$ where
$*=+$ for $i$ odd and $*=-$ when $i$ is even, the composition factors
are labelled by $i^+, (i-1)^-, (i-2)^+, \ldots\ $.
The costandard module $\nabla(i^-)$ is described similarly.
The costandard module
$\nabla(i^{\pm})$ has socle $i^{\pm}$.
The top of $\nabla(i^{\pm})$ is $1^{\pm}$ if $i$ is odd and is 
$1^{\mp}$ otherwise.

%

\begin{prop}
The algebra $D_n$ is quasi-hereditary with weight set

$\{1^+,1^-,2^+, 2^-,\dots,n^+, n^-\}$ and order
\begin{eqnarray*}
i^+<j^{\pm} \Longleftrightarrow i<j, \\
i^-<j^{\pm} \Longleftrightarrow i<j.
\end{eqnarray*}
\end{prop} 


In general,
the standard $D_n$-module $\Delta(i^{\pm})$ is the serial module of
length $i$ with socle $L(1^{\pm})$, and the signs of 
the labels of its composition factors are either all $+$ or all $-$. 
The $\nabla(i^\pm)$ are also uniserial and 
have alternating signs on their composition
factors.

Now we define the tilting modules. 
Recall that a  module has a $\Delta$-filtration if it has a filtration
whose successive quotients are isomorphic to standard
modules, similarly one  defines a $\nabla$-filtration.
Recall also that for each $i^{\pm}$ there is a unique  indecomposable 
 module, which we denote by $T(i^{\pm})$, which
has both a $\Delta$-filtration and a $\nabla$-filtration, with one
composition factor of the form $L(i^{\pm})$ and all other composition
factors of the form $L(j^{\pm})$ with $j<i$ (both $L(j^+)$ and
$L(j^-)$ may appear). We refer to direct sums of $T(i^{\pm})$ as tilting modules.
Then  for $n$ odd we have
$T(n^{\pm})=P(1^{\pm})$, and for $n$ even $T(n^{\pm}) = P(1^{\mp})$, 
and these two are both projective and injective. 

\begin{remark}\label{rem:T(i)-sequ}
Note that $T((i-1)^+)$ is a quotient of $P(1^{\eps})$ by 
$P(i^-)$, where $\eps=+$ if 
$i$ is even and $\eps=-$ if $i$ is odd. The same is 
true for $T((i-1)^-)$ with signs exchanged. So we have 
a short exact sequence
\[
0\to P(i^{\mp})\to P(1^{\eps})\to T((i-1)^{\pm})\to 0
\]
(with appropriate sign $\eps$). 
\end{remark}

We will also use $Q(1^{\pm})$ to denote the injective hull of $L(1^{\pm})$. 
Note that $Q(1^{\pm})\cong P(1^{\pm})$ if $n$ is odd and 
$Q(1^{\pm})\cong P(1^{\mp})$ if $n$ is even.  

%
\section{$\Delta$-filtered modules for $D_n$}\label{s:F(D)}
%

Let ${\mathcal F}(\Delta)$ be the class of all $\Delta$-filtered
$D_n$-modules. 
In this section, we describe the 
properties of ${\mathcal F}(\Delta)$ 
extending the results of ~\cite{bhrr} to the algebra $D_n$. 

Recently, R. Tan has studied the category of $\Delta$-filtered modules 
for the Auslander algebra $E$ of a self-injective Nakayama algebra, 
in particular the submodules of the projective $E$-modules. 
In Section 3 of~\cite{tan}, 
she obtains similar results to those we explain here. 

%
\begin{lemma}\label{lm:submodules}
${\mathcal F}(\Delta)$ is closed under taking submodules. 
\end{lemma}
\begin{proof} 
Since submodules of standard modules are also standard modules, this
follows
easily by 
induction over filtration length. 
%
%
%
%
\end{proof}

We say that the socle of a $D_n$-module $M$ {\it is 
generated by $L(1^{\pm})$} if the only simples 
appearing in the socle are $L(1^+)$ and/or $L(1^-)$. 
%
%
Similarly, we say that the top of $M$ {\em is generated 
by $L(1^{\pm})$} if the only simples
appearing in the top are $L(1^+)$ and/or $L(1^-)$. 

\begin{lemma}\label{lm:socle}
${\mathcal F}(\Delta)$ is the set of all modules 
with socle generated by $L(1^{\pm})$. 
%
\end{lemma}
\begin{proof}
This is similar to Lemma 7.1 of~\cite{dr}. 

It is clear that the socle of any module in ${\mathcal F}(\Delta)$ 
is of that form because all standard modules 
have $L(1^+)$ or $L(1^-)$ as 
socle. Now we show that any $D_n$-module with socle 
generated by $L(1^{\pm})$ is in ${\mathcal F}(\Delta)$

\noindent
Assume that 
\[\soc(M)=\bigoplus_{j\in J_1} L(1^+)
\oplus\bigoplus_{j\in J_2} L(1^-).
\] 
Then we have an embedding 
\[
M\hookrightarrow 
\bigoplus_{j\in J_1}I(1^+)\oplus\bigoplus_{j\in J_2} I(1^-).
\] 
Now $I(1^+)$ and $I(1^-)$ are tilting modules, 
so in particular, they are $\Delta$-filtered and 
hence 
$M$ is a submodule of a $\Delta$-filtered module. 
Therefore, $M\in {\mathcal F}(\Delta)$ by Lemma~\ref{lm:submodules}. 
\end{proof}

Similarly, the class of $\nabla$-filtered modules 
${\mathcal F}(\nabla)$ is the set of modules with top 
generated by $L(1^{\pm})$. 

\begin{lemma}\label{lem:pdone}
The modules in $\mathcal F(\Delta)$ are the 
$D_n$-modules with projective dimension $\le 1$.
\end{lemma}

\begin{proof} 
%
This follows in a similar fashion to \cite[Lemma 1]{bhrr}.
\end{proof}
Next we observe that we can describe the 
submodules of $P(1^{\pm})$ similarly as in Lemma 2 
of~\cite{bhrr}. 
\begin{lemma}\label{lm:socle-submod} 
Let $M$ be a $D_n$-module. 
Then the following are equivalent: \\
(i) $M$ is a nonzero submodule of $P(1^{\pm})$. 

\noindent 
(ii)  
$
\soc(M)=\left\{
\begin{array}{ll}
L(1^{\pm}) & \mbox{ if $n$ is odd} \\
L(1^{\mp}) & \mbox{ if $n$ is even.}
\end{array}\right.
$
\end{lemma}

\begin{proof} 
Follows from the fact that for odd $n$, 
$P(1^{\pm})$ is the injective envelope of $L(1^{\pm})$ 
and for even $n$ it is the injective envelope of 
$L(1^{\mp})$. 
\end{proof}

\begin{lemma}\label{lm:indecs}
Any nonzero submodule of $P(1^{\pm})$ is  
indecomposable and belongs to ${\mathcal F}(\Delta)$.
\end{lemma}

\begin{proof} 
Let $M$ be a nonzero submodule of $P(1^{\pm})$. 
Since $P(1^{\pm})$ is equal to the injective 
envelope of $L(1^{\eps})$ (for $\eps=\pm$ if $n$ is 
odd and $\eps=\mp$ if $n$ is even), $M$ is indecomposable. 
By Lemma~\ref{lm:submodules}, any submodule of 
$P(1^{\pm})$ is in ${\mathcal F}(\Delta)$. 
\end{proof}

%
\section{Submodules and quotients of the projective modules}\label{s:D(I)}
%

In this section, we are going to describe certain 
indecomposable $\Delta$-filtered modules, the $\Delta(I)$. 
They are submodules of $P(1^+)$ or $P(1^-)$  
and do not have self-extensions. 
These form the key components of our extensions $E(I,J)$. 


%
%

From now, we abbreviate
$\{1,2,\dots,n\}$ by $[n]$ and
$\{1^+,1^-,\dots,n^+,n^-\}$ by $[n]^{\pm}$.
Unless mentioned 
otherwise we will always assume that a subset 
$I=\{i_1,i_2\dots,i_k\}\subset[n]$ is decreasingly 
ordered, i.e. $i_1>i_2>\dots>i_k$.
We call a subset $I$ of $[n]^{\pm}$
\emph{signed}
if there is no $1\le i\le n$ with both
$i^+\in I$ \emph{and} $i^-\in I$, i.e. 
$I=\{i_1^{\epsilon_1},\dots,i_k^{\epsilon_k}\}$,
for some subset $\{i_1,\dots,i_k\}$ of $[n]$ and $k$, $1\le k\le n$
with $\epsilon_l\in \{+,-\}$ for $l=1,\dots,k$.
Now let $I=\{i_1^{\epsilon_1},\dots,i_k^{\epsilon_k}\}$ be
a signed subset of $[n]^{\pm}$.
\begin{itemize}
\item
If for $j=1,\dots,k-1$ we have
$\epsilon_j\neq\epsilon_{j+1}$ we say that the signs
$\epsilon_1,\dots,\epsilon_k$ of $I$
are \emph{alternating} and we also call
$I$ an \emph{alternatingly signed subset}.
\item
If $\epsilon_j\neq\epsilon_{j+1}$ if and only if $i_{j+1}-i_j$ is
even, we say that the signs are \emph{step-alternating} and we also
call $I$ \emph{step alternatingly signed}. 
\end{itemize}
%
Recall that, for any sign $\epsilon$,
$\bar{\epsilon}$ is the sign opposite to $\epsilon$.
Let $s(a)$ be the sign of a number $a$, 
i.e. $s(1)=+$ and $s(-1)=-$. 
We
can now define the modules $\Delta(I)$:
\begin{definition}\label{defn:Delta}
Let $I=\{i_1^{\epsilon_1},i_2^{\epsilon_2},
\dots,i_k^{\epsilon_k}\}$ be a signed subset of
$[n]^{\pm}$ with
$i_1>i_2>\dots>i_k$.
\begin{itemize}
\item[(i)]
Assume that
the signs are alternating. Then we set
$\Delta(I)$ to be the submodule of
$$\left\{
\begin{array}{ll}
P(1^+) \mbox{ with $\Delta$-support $I$} & \mbox{if }
s((-1)^n)=\overline{\epsilon_k} \\
P(1^-) \mbox{ with $\Delta$-support $I$} & \mbox{if }
s((-1)^n)=\epsilon_k\, .
\end{array}
\right.$$ 
\item[(ii)]
Assume that the signs are step-alternating. Then we
set $\nabla(I)$ to be the factor module of
$$\left\{
\begin{array}{ll}
P(1^+) \mbox{ with $\nabla$-support $I$} & \mbox{if }
s((-1)^{i_k})=\overline{\epsilon_k} \\
P(1^-) \mbox{ with $\nabla$-support $I$} & \mbox{if }
s((-1)^{i_k})=\epsilon_k\, .
\end{array}
\right.$$
\end{itemize}
\end{definition}
It is clear that $\Delta(I)$ is unique (the existence can be
seen using the quiver and relations),
i.e. there can be no two different submodules of $P(1^+)$ 
with same $\Delta$-support. Is it also clear that a submodule 
of $T(n^{\pm})$ is uniserial in its $\Delta$-filtration.

\begin{remark}\label{rem:|submod|}
In other words, for any signed subset $I$ with
signs $\{\epsilon_1,\dots,\epsilon_k\}$ we have
the following:

\begin{enumerate}
\item[(i)]
If $I$ is alternatingly signed,
then $\Delta(I)$ is a submodule of
$P(1^{\epsilon_k})$ for odd $n$ and of
$P(1^{\overline{\epsilon_k}})$ if $n$ is even.
In all other cases, we do not define $\Delta(I)$.
There will be other $\Delta$-filtered modules with $\Delta$-support equal to
$I$ but these modules will not be submodules of a single projective
module.
%
\item[(ii)]
We similarly only define $\nabla(I)$ 
if $I$ is step-alternatingly signed. 
\end{enumerate}
\end{remark}

From now on we will use the following convention: 
If we say that $I$ is signed and we are working with a module
$\Delta(I)$ then we most of the time tacitly assume that the set $I$ is alternatingly
signed. Similarly if we work with $\nabla(I)$ then the set is step
alternatingly signed. 

\medskip

Now we describe the relation between submodules of
$P(1^{\pm})$ and subsets of $[n]$ and between 
factor modules of $P(1^{\pm})$ and subsets of $[n]$.

\begin{lemma}
\begin{enumerate}
\item[(i)]
The map sending a submodule $M$ of $P(1^{+})$ 
(or $P(1^-)$) to its
$\Delta$-support induces a bijection between the submodules of
$P(1^+)$ (or $P(1^-)$)
and the subsets of $[n]$ (ignoring signs). 
\item[(ii)]
The map sending a quotient module $N$ of $P(1^+)$ (or $P(1^-)$)
to its $\nabla$-support induces a bijection between the factor 
modules of $P(1^+)$ (or $P(1^-)$) and the subsets 
of $[n]$ (ignoring signs).
\end{enumerate}
In particular, $P(1^+)$ and $P(1^-)$ each have precisely $2^n$
submodules and $2^n$ factor modules. 
\end{lemma}

\begin{proof}
It is enough to consider (i). Let $M$ be a submodule of $P(1^+)$. 
By Remark~\ref{rem:|submod|}, the map induces a bijection between 
$P(1^+)$ and 
the (alternatingly) signed subsets
$\{i_1^{\epsilon_1},\dots,i_k^{\epsilon_k}\}$
of $[n]^{\pm}$ with
$\epsilon_k=\overline{s((-1)^n)}$. But this is in bijection to 
the subsets
$\{i_1,\dots,i_k\}\subset[n]$. 
\end{proof}

In what follows, we will need to go from a
subset of $[n]$ to a signed subset of $[n]^{\pm}$:
If we associate to
$I_0=\{i_1,\dots,i_k\}\subset$ $[n]$ a $k$-tuple
$\keps=\{\epsilon_1,\dots,\epsilon_k\}$ of signs,
we will call the resulting
$I= \{i_1{^\epsilon_1},\dots,i_k{^\epsilon_k}\}$
a \emph{signed version} of
$I_0$ and we say that $I_0$ is the 
\emph{unsigned version} of $I$.

\begin{lemma}\label{lm:signed}
Let $I_0$ be a non-empty subset of $[n]$.
Then there are unique signed versions
$I$ and $I'$ of $I_0$
such that
$\Delta(I)$ is a submodule of $P(1^+)$
and $\nabla(I')$ is a factor module
of $P(1^{+})$.
\end{lemma}
Note that the same statements hold for $P(1^-)$ with
``opposite'' signs.
We leave the (easy) proof to the reader.



Let $I$ and $J$ be signed subsets of $[n]^{\pm}$.
By abuse of terminology
we say that $J$ is a \emph{complement} to $I$ if their unsigned versions
$I_0$ and $J_0$ are such that $J_0 = [n]\setminus I_0$.
Clearly, the complement to a signed subset is \emph{not} unique.  
But we have the following, which is easy to prove, and is analogous
to [page 298]\cite{bhrr}. 
\begin{lemma}\label{lm:s-e-s}
Let $I$ be a signed subset of
$[n]^{\pm}$.
Assume 
that $\Delta(I)$ is a submodule of $P(1^+)$.
Then there is a unique complement $I^c$ of 
$I$ such that there is a short exact sequence
$$
0\to\Delta(I)\to P(1^+)\to\nabla(I^c)\to 0.
$$
\end{lemma}

\section{Results relating $\Ext^{\bullet}_{A_n}$ to
  $\Ext^{\bullet}_{D_n}$}\label{indres}
%

We have seen that
$D_n$ is a skew group ring over $A_n$ which
allows us to relate the $\Delta$-filtered
modules of these two algebras.
In this section, we use induction and restriction to relate $\Ext_{A_n}$ 
and $\Ext_{D_n}$. 

Since $D_n$ is free as module over $A_n$,
the  adjoint functors given by the $A_n, D_n$ bimodule $D_n$, that is,
inducing and corestricting, have good properties:
They preserve projectives, so we have Shapiro's Lemma, 
$
\Ext^{\bullet}_{D_n} (X\otimes_{A_n}D_n, Y) 
\cong \Ext^{\bullet}_{A_n}(X, Y\downarrow _{A_n})\, . 
$
(see for example~\cite[2.8.4]{benson}). 
Furthermore, every $D_n$-module $X$ is relative $A_n$-projective, 
that is, the multiplication map $X\otimes_{A_n}D_n \to X$ splits 
(by a 'Maschke-type'
argument), using ${\rm char}(k)\neq 2$.

In Section \ref{Dnprops} we have defined a partial order 
on the labels for the simple modules, 
and have seen that $D_n$ is quasi-hereditary with respect to this order.

\begin{lemma}\label{lemind}
For each $i^\epsilon$ with $\epsilon=+$ or $\epsilon=-$ we have
\begin{enumerate}
\item[(a)]
$\Delta(i^\epsilon)\downarrow _{A_n} \cong \Delta(i)$ and
$\nabla(i^\epsilon)\downarrow_{A_n} \cong \nabla(i)$.
\item[(b)]
$\Delta(i)\otimes_{A_n}D_n\cong \Delta(i^+)\oplus \Delta(i^-)$
and $\nabla(i)\otimes_{A_n}D_n \cong \nabla(i^+)\oplus \nabla(i^-)$.
\item[(c)]
Suppose $X$ is any $A_n$-module. Then $X\in \mathcal{F}(\Delta_{A_n})$
if and only if $X\otimes_{A_n}D_n$ belongs to 
$\mathcal{F}(\Delta_{D_n})$.
\end{enumerate}
\end{lemma}

\begin{proof}
Part (a) is easily seen directly.

(b) We know that the multiplication
map $\Delta(i^\epsilon)\otimes_{A_n}D_n \to \Delta(i^\epsilon)$
splits. Using 
part (a), we get that $\Delta(i)\otimes_{A_n}D_n$ has a direct summand
isomorphic to $\Delta(i^+)$ and also a direct summand isomorphic to $\Delta(i^-)$. Hence by dimensions, $\Delta(i)\otimes_{A_n}D_n$ must be the
direct sum as stated in (b).

(c) For a module $X$ of any quasi-hereditary algebra $\Lambda$,
it is known that $X\in {\mathcal F}(\Delta)={\mathcal F}(\Delta_{\Lambda})$ 
if and only if
$\Ext^1(X, \nabla(j))=0$ for all $j$ 
(\cite[appendix A]{Don98}). 

Now take $\Lambda$ to be $A_n$ or $D_n$, and use Shapiro's Lemma 
and part (a),
$$\Ext_{A_n}^1(X, \nabla(j)) \cong
\Ext_{D_n}^1(X\otimes_{A_n}D_n, \nabla(j^\epsilon))
$$
So $X$ has a $\Delta$-filtration if and only if the induced module 
$X\otimes_{A_n}D_n$ has a $\Delta$-filtration.
\end{proof}


\begin{cor}
$$
  \Ext_{A_n}^{\bullet}(\Delta(i), \Delta(j))
\cong
  \Ext_{D_n}^{\bullet}(\Delta(i^+), \Delta(j^{\epsilon}))
\oplus
  \Ext_{D_n}^{\bullet}(\Delta(i^-), \Delta(j^{\epsilon}))
$$
for a sign $\epsilon$.
\end{cor}
\begin{proof}
Since $\Delta(j^{\epsilon})\downarrow_{A_n} \cong \Delta(j)$, 
by applying Shapiro's lemma:
$
  \Ext_{A_n}^{\bullet}(\Delta(i), \Delta(j))
\cong
  \Ext_{D_n}^{\bullet}(\Delta(i)\otimes_{A_n}D_n, \Delta(j^{\epsilon})).
$
This is isomorphic to
$
  \Ext_{D_n}^{\bullet}(\Delta(i^+), \Delta(j^{\epsilon}))
\oplus
  \Ext_{D_n}^{\bullet}(\Delta(i^-), \Delta(j^{\epsilon}))
$
using Lemma \ref{lemind}~(b).
\end{proof}

Suppose $I$ is an (alternatingly) signed subset of $[n]^\pm$.
We let $-I$ denote the signed set which has the same underlying
unsigned set $I_0$ as $I$ but with opposite signs to $I$.
That is, $i^\epsilon \in I$ if and
only if $i^{\bar{\epsilon}} \in -I$.
Recall that we use $I_0$ for the unsigned version of $I$.
\begin{prop}\label{prop:ext}
For $I$ and $J$ 
signed subsets of 
$[n]^\pm$ we have:
\begin{enumerate}
\item[(a)]
$\Delta(I)\downarrow _{A_n} \cong \Delta(I_0)$ 
\item[(b)]
$\Delta(I_0)\otimes_{A_n}D_n\cong \Delta(I)\oplus \Delta(-I)$
\item[(c)]
$
  \Ext_{A_n}^{\bullet}(\Delta(I_0), \Delta(J_0))
\cong
  \Ext_{D_n}^{\bullet}(\Delta(I), \Delta(J))
\oplus
  \Ext_{D_n}^{\bullet}(\Delta(-I), \Delta(J))\, .
$
\end{enumerate}
\end{prop}
\begin{proof}
(a) The module $\Delta(I)\downarrow _{A_n}$
has a $\Delta$-filtration by Lemma \ref{lemind}~(a)
and induction on filtration length.
It also has $\Delta$-support equal to $I_0$ as an $A_n$-module.
Since restriction is exact 
 $\Delta(I)\downarrow _{A_n}$ remains a submodule of 
 $P(1^\epsilon) \downarrow _{A_n}$ ($\epsilon$ of the appropriate
sign)
and hence $\Delta(I)\downarrow _{A_n}$ is a submodule of $P(1)$.
Thus  $\Delta(I)\downarrow _{A_n} \cong \Delta(I_0)$ as this is the
only submodule of $P(1)$ with the same $\Delta$-support. 

(b) 
Using part (a) we may argue similarly to the proof of 
Lemma \ref{lemind}~(b)
to show that both $\Delta(I)$ and $\Delta(-I)$ are direct
summands of $\Delta(I_0) \otimes_{A_n} D_n$ and hence, by dimensions,
this tensor product must be equal to the direct sum.

(c) This result follows as in the proof of the previous corollary.
\end{proof}

%
\section{Extensions between the $\Delta(I)$}\label{s:exts}
%

We are now ready to study the extensions between 
two $\Delta$-filtered modules $\Delta(I)$ and 
$\Delta(J)$. 
In this section, we give a formula for the dimension of 
$\Ext^1(\Delta(I),\Delta(J))$ and $\Hom(\Delta(I),\Delta(J))$ 
over both $A_n$ and $D_n$. 

Unless stated otherwise, homomorphism and extension 
spaces are taken over $D_n$. 
%
We will furthermore write $\hom(A,B)$ for $\dim\Hom(A,B)$ 
and $\ext^1(A,B)$ for $\dim\Ext^1(A,B)$. 

\begin{lemma}\label{lem:ext1}
In $A_n$: 
for $I_0$ and $J_0$ unsigned subsets of $[n]$, we have
$$
\ext^1_{A_n}(\Delta(I_0),\Delta(J_0))=\\
\hom_{A_n}(\Delta(I_0),\Delta(J_0))-
\hom_{A_n}(\Delta(I_0),P(1))+
\hom_{A_n}(\Delta(I_0),\nabla(J_0^c)).
$$

And in $D_n$:
for $I$ and $J$ signed subsets of $[n]^\pm$, we have
$$
\ext_{D_n}^1(\Delta(I),\Delta(J))=\\
\hom_{D_n}(\Delta(I),\Delta(J))-
\hom_{D_n}(\Delta(I),P(1^\epsilon))+
\hom_{D_n}(\Delta(I),\nabla(J^c))
$$
where $\epsilon$ is the sign of the largest element in $J$ if $n$ is
odd and the opposite sign if $n$ is even.
\end{lemma}
\begin{proof}
We prove the signed version --- the unsigned version follows
similarly.
This lemma follows by 
applying $\Hom_{D_n}(\Delta(I),\--)$ to the following short exact 
sequence 
\[
0\to \Delta(J)\to P(1^{\epsilon}) 
\to\nabla(J^c) \to 0
\]
from Lemma~\ref{lm:s-e-s},
and noting that 
 $\Ext_{D_n}^1(\Delta(I),P(1^{\epsilon}))=0$ 
as $P(1^\epsilon)$ is a tilting module.
\end{proof}

All the terms on the right hand side of the expression in \ref{lem:ext1}  are
calculable.
We will get the first term in Proposition~\ref{prop:hom-I-J}. 

The second term is the sum (\cite[appendix A]{Don98})
$$\sum_{i^\delta\in [n]^\pm} 
\dim_{\Delta}\Delta(I)_{i^\delta}
\dim_{\nabla}P(1^\epsilon)_{i^\delta}
.$$ 

The $\nabla$-support of $P(1^{\epsilon})$ is 
$\{n{^\epsilon}, (n-1)^\epsilon,\ldots, 1^\epsilon\}$ for $n$ odd and
$\{n^{\bar{\epsilon}}, (n-1)^{\bar{\epsilon}},\ldots, 
1^{\bar{\epsilon}}\}$ for $n$ even 
and so this sum is given by
$|I \cap \{n^{\epsilon}, (n-1)^\epsilon,\ldots, 1^\epsilon\}|$ if $n$
is odd and
$|I \cap 
\{n^{\bar{\epsilon}}, (n-1)^{\bar{\epsilon}},\ldots, 
1^{\bar{\epsilon}}\}|$ for $n$ even.

The third term is given by the sum
$$\sum_{i^\delta\in [n]^\pm} 
\dim_{\Delta}\Delta(I)_{i^\delta}
\dim_{\nabla}\nabla(J^c)_{i^\delta}
$$ 
and this is equal to the number of elements that are 
both in $I$ and in $J^c$, i.e. 
$|I\cap J^c|$.

\begin{prop} \label{prop:I-in-J}
$\ext^1_{D_n}(\Delta(I),\Delta(J))=0$ 
for all $I_0\subset J_0$ or 
$J_0\subset I_0$. 
\end{prop}
\begin{proof}
This follows from the result for the unsigned sets $I_0$ and $J_0$
in \cite{bhrr}
and Proposition \ref{prop:ext}~(c). 
\end{proof}


\begin{cor}\label{cor:support(M)-I}
Suppose $M \in  \doog$ with $\Delta$-support $J$. Then
$\ext^1_{D_n}(\Delta(I),M)=0$ 
and 
$\ext^1_{D_n}(M,\Delta(I))=0$ 
if $J_0 \subset I_0$.
\end{cor}
\begin{proof}
By the previous lemma
$\ext^1_{D_n}(\Delta(I),\Delta(j^\epsilon))=0$ 
and 
$\ext^1_{D_n}(\Delta(j^\epsilon),\Delta(I))=0$ for
$j \in I_0$. Induction on the $\Delta$-filtration length of $M$ then gives the
result.
\end{proof}

We now focus on calculating  $\Hom_{A_n}(\Delta(I), \Delta(J))$ and 
$\Hom_{D_n}(\Delta(I), \Delta(J))$.

Let us start introducing the necessary notation first. 
Let $I_0, J_0$ be subsets of $[n]$, with
$I_0 = \{ i_1 > i_2 > \ldots >i_r\}$ and
similarly $J_0=\{ j_1 > j_2 > \ldots > j_s\}$.

Call a subset $K$,  of $I_0$ an {\em initial segment} if it is of 
the form
$K:= \{ i_{r-u} > i_{r-u+1} > \ldots > i_r\}$ for some $u\le |I|$ 
(so in total there are $|I|$ nonempty initial segments).

Now define an {\em order $\leq$} 
on the subsets of $[n]$.
Let $V, W$ be such subsets, say $V=\{ v_1 > v_2 > \ldots > v_x\}$
and $W=\{ w_1 > w_2 > \ldots > w_y\}$. Then set
$V \leq W$ if and only if
$$\begin{array}{ll}
 & x\le y \mbox{ (ie $|V|\le |W|$)} \\ 
\mbox{and } & 
v_1 \leq w_1, v_{2} \leq w_{2}, \ldots, v_x\leq w_{x}.
\end{array}$$

\begin{prop}\label{prop:dim-hom}
Let $I_0=\{i_1>i_2>\dots>i_r\}$ and $J_0=\{j_1>j_2>\dots>j_s\}$ 
be subsets of $[n]$. 
The dimension of $\Hom_{A_n}(\Delta(I_0), \Delta(J_0))$ is equal
to the number of initial segments $K$ of $I_0$ such that
$K\leq J_0$.
\end{prop}

\begin{proof} \ For the moment we write $I=I_0$ and $J=J_0$. 
To obtain a homomorphism from $\Delta(I)$ to $\Delta(J)$ 
we need to map a factor module of $\Delta(I)$ to 
a submodule of $\Delta(J)$. Factor modules of $\Delta(I)$ 
which also embed in $\Delta(J)$ must have a $\Delta$-filtration
by Lemma \ref{lm:submodules} and
hence are given by initial segments $I_u$. Now 
$\Delta(I_u)$, $I_u=\{i_{r-u+1}>i_{r-u+2}>\dots>i_r\}$ embeds 
as a submodule in $\Delta(J)$ if and only if $i_{r-u+1}\le j_1$, 
$i_{r-u+2}\le j_{2}$, $\dots, i_r\le j_{u}$. 
\end{proof}

\begin{ex}\label{ex:I-P1}
If $J_0=\{ n,n-1, \ldots, 1\}$ then $\Delta(J_0) =P(1)$ and
all initial segments have the required property and so
the dimension of $\Hom_{A_n}(\Delta(I_0),P(1))$ is $|I_0|$. 
\end{ex}

The signed version of Proposition~\ref{prop:dim-hom} is 
then:
\begin{prop}\label{prop:hom-I-J}
Let $I=\{i_1>\dots>i_s\}$ and 
$J=\{j_1>\dots>j_s\}$ be 
signed subsets of $[n]^{\pm}$. 
The dimension of $\Hom(\Delta(I), \Delta(J))$ 
is equal to the number of initial segments $K$ of $I$ such 
that
$K\leq J$ and such that the sign of $i_{r-u+1}$, the first element 
in $K$, is equal to 
the sign of $j_1$. 
\end{prop}
\begin{proof}
This follows from Proposition~\ref{prop:dim-hom} and 
the fact that there is a homomorphism $\Delta(i_{r-u+1}^{\eps})$ 
$\to\Delta(j_1^{\delta})$ if and only if $\eps=\delta$ and $i_{r-u+1}\leq
j_1$.
\end{proof}

\begin{ex}\label{ex:I*-P1}
If $i<j$ for all $i\in I_0$ and all $j\in J_0$ and 
$\Delta(J)$ is a submodule of 
$P(1^{\gamma})$ then we have 
\[
\hom(\Delta(I), \Delta(J))
=\hom(\Delta(I), P(1^\gamma))\, .
\]
\end{ex}
%

%
\section{Ext-result with $m$ gaps}\label{s:gaps}
%

In this section we calculate the extension group between $\Delta(I)$ and
$\Delta(J)$ where the underlying unsigned sets for $I$ and $J$ have ``$m$ gaps''.
We also find a $\Delta$-filtered module with no self-extensions that is an 
extension of $\Delta(I)$ by $\Delta(J)$. 
This is the module we will use to build up the $M({\bf d})$ in 
Section~\ref{s:construction}. 
Since we are interested in associating 
such  modules to Richardson orbits we may assume that all the ``gaps''
only occur on one side of $I$ 
and that $I$ and $J$ satisfy a symmetry condition, so that
$I=\Phi(J)$ defined below.

We will continue to use the notation $I_0$ and $J_0$ for the unsigned
versions of the signed subsets $I$ and $J$ of $[n]^{\pm}$.
We now define a map $\Phi$ on both signed and unsigned sets.
Let $I_0$ be a unsigned subset of $[n]$. We define 
$$\Phi(I_0) = \{ n-i+1\mid i \in I_0 \}.$$
We now define $\Phi(I)$ to be $\Phi(I_0)$ with signs chosen so that
the largest element of $\Phi(I)$ has opposite sign to that of the
largest element of $I$.

We now fix a signed subset $I$ of $[n]^\pm$ where 
the sign of the largest element in $I$ is $+$ so that $\Delta(I)$
has simple socle $L(1^+)$.
We set $J = \Phi(I)$ and note that 
the sign on the largest element in $J$ 
is $-$.
Let $[n] \setminus I_0 = \{ a_1, a_2, \ldots, a_m\}$ (in decreasing
order) and let
$b_j = n+1 - a_j$ for $1\le j \le m$ so that 
$[n]\setminus J_0 = \{ b_m,b_{m-1}, \ldots, b_1\}$.
We note that if $i \in I_0 \setminus J_0$ then $n+1-i \in
J_0 \setminus I_0$.  
We impose a further condition that 
if $i \in I_0 \setminus J_0$ then $i \ge \frac{n+1}{2}$.  

We then choose signs  $\epsilon_i$ and $\delta_i$ so that 
$I^c= \{a_1^{\bar{\epsilon}_1}, a_2^{\bar{\epsilon}_2},\ldots,
a_m^{\bar{\epsilon}_m}\}$ and 
$J^c= \{b_m^{\bar{\delta}_m}, b_{m-1}^{\bar{\delta}_{m-1}},\ldots,
b_1^{\bar{\delta}_1}\}$. 
We thus have short exact sequences
$$0 \to \Delta(I) \to Q(1^+) \to \nabla(I^c) \to 0, \ \ \mbox{ and } \  \  
0 \to \Delta(J) \to Q(1^-) \to \nabla(J^c) \to 0$$  
where $Q(1^\epsilon)$ is the injective hull of $L(1^\epsilon)$.
\begin{lemma}\label{lem:extAIJ}
We have $\ext_{A_n}(\Delta(I_0), \Delta(J_0)) = 0 $ and 
$\ext_{A_n}(\Delta(J_0), \Delta(I_0)) =|J_0 \cap I_0^c|$. 
\end{lemma}
\begin{proof}
This is a matter of calculating the right 
hand side in Lemma \ref{lem:ext1}.

Now, $\hom_{A_n}(\Delta(I_0),\Delta(J_0))$ is the number of
overlapping segments. We claim we have $|I_0\cap J_0|$
overlapping segments.
We let $l$ be minimal such that
$a_l \in J_0 \setminus I_0$. The assumptions of $I_0$ and $J_0$ imply
that $a_l \le \frac{n+1}{2}$.
Now the last such overlapping segment is: 
\setcounter{MaxMatrixCols}{15}
$$
\begin{matrix}
\cdots& i_{s-t} &\cdots &  i_{s-1} & i_s & \cdots & i_{n-m} & &  &\\
 &j_1 &\cdots &  j_t & a_l & \cdots & j_{u-1} & j_{u} &\cdots
&j_{n-m}
\end{matrix}
$$
where $s, t, u$ are appropriate integers and 
$i_{s-1}> a_l > i_s$ (note we cannot have equality as $a_l \not \in
I_0$).
It is clear that we cannot get any more overlapping segments
as $i_s > a_1$.
The total number of overlapping segments is thus the amount of overlap
in the above diagram. 
For the calculation, let $r= |J_0\cap I_0^c|$, the number
of gaps in $I_0$ which are not gaps in $J_0$, which is
also equal to $|I_0\cap J_0^c|$. The amount of overlap in 
the above diagram  is equal to:
\begin{align*}
&|\{ i \in I_0 \mid i < a_l \}| +|\{ j \in J_0 \mid j > a_l \}|\\
&= a_l - 1 -\#\mbox{gaps after $a_l$ in $I_0$} + n - a_l -\#\mbox{gaps before
    $a_l$ in $J_0$} \\ 
&= n - 1 - (m -l) - |\{i \in [n] \mid i \not\in I_0 \mbox{ and } i
  \not \in J_0 \mbox{ and } i > a_l \}|
- |\{i \in [n] \mid i \in I_0 \mbox{ and } i \not \in J_0  \mbox{ and } i > a_l\}| \\
&= n - 1 -m +l - (l-1) - r \\
&= n  -m -r 
\end{align*}
(we have used that $I_0\cap J_0$ has only weights $< (n+1)/2$). 
Now $|I_0 \cap J_0| = |I_0| - |I_0\setminus J_0 | = n-m - |I_0 \cap
J_0^c|=  n -m  -r$.
Thus  $\hom_{A_n}(\Delta(I_0),\Delta(J_0))=|I_0 \cap J_0|$.
Now
$\hom_{A_n}(\Delta(I_0),P(1))
= |I_0| = n-m$
and 
\ $\hom_{A_n}(\Delta(I_0),\nabla(J_0^c))
= |I_0 \cap J_0^c| = r.$
Thus 
$$\ext^1_{A_n}(\Delta(I_0),\Delta(J_0))=n-m-r -(n-m) + r=0.$$

To calculate the other $\Ext$ group we consider,
 $\hom_{A_n}(\Delta(J_0),\Delta(I_0))$ which is the number of
 overlapping segments. By construction, $j_s \le i_s$ for all $s$,
 thus $\Delta(J_0)$ in fact embeds in $\Delta(I_0)$ and the number of
 overlapping
segments is $|J_0|= n-m$. 
Thus 
$\hom_{A_n}(\Delta(J_0),\Delta(I_0))=n-m.$ 
We also have
$\hom_{A_n}(\Delta(J_0),P(1))
= |J_0| = n-m.
$
Hence 
$\hom_{A_n}(\Delta(J_0),\Delta(I_0))=\hom_{A_n}(\Delta(J_0),P(1))$
and 
$\ext^1_{A_n}(\Delta(J_0),\Delta(I_0))=
\hom_{A_n}(\Delta(J_0),\nabla(I_0^c))$.
Now
$\hom_{A_n}(\Delta(J_0),\nabla(I_0^c))
= |J_0 \cap I_0^c |
= |I_0 \cap J_0^c |$ 
thus 
$$\ext^1_{A_n}(\Delta(J_0),\Delta(I_0))
= |I_0 \cap J_0^c |.\qedhere$$
\end{proof}

We now prove the following signed version of the above lemma.
\begin{lemma}\label{lem:extBIJ}
Let $|J_0 \cap I_0^c| =r$.
We have $\ext_{D_n}^1(\Delta(I), \Delta(J)) = 0 =
\ext_{D_n}^1(\Delta(I), \Delta(-J))$, 
$$\ext_{D_n}^1(\Delta(J), \Delta(I)) = 
\begin{cases}\frac{r}{2} & \mbox{if $r$ even} \\
\frac{r+1}{2} & \mbox{if $r$ odd} 
\end{cases}$$ and
$$\ext^1_{D_n}(\Delta(-J),\Delta(I))=
\begin{cases}\frac{r}{2} & \mbox{if $r$ even} \\
\frac{r-1}{2} & \mbox{if $r$ odd.} 
\end{cases}$$ 
\end{lemma}
\begin{proof}
Since  
$$\ext^1_{D_n}(\Delta(I), \Delta(J))+
\ext^1_{D_n}(\Delta(I), \Delta(-J))  
 = \ext^1_{A_n}(\Delta(I_0), \Delta(J_0))= 0$$
using Proposition \ref{prop:ext}~(c) and
the previous lemma, we have the first part of the lemma.

For the second part, recall from the proof of Lemma 9.1 that
$${\Ext}^1(\Delta(J_0), \Delta(I_0))\cong {\Hom}(\Delta(J_0), \nabla(I_0^c))
$$
Therefore we have, using Lemma 7.3 that
$${\Ext}^1_{D_n}(\Delta(J)\oplus \Delta(-J), \Delta(I))\cong {\Hom}_{D_n}(\Delta(J)\oplus \Delta(-J), \nabla(I^c))
$$
and this has dimension $|I_0\cap J_0^c|=r$.

We have a surjection ${\Hom}(\Delta(J), \nabla(I^c)) \to {\Ext}^1(\Delta(J), \Delta(I))$, 
and a similar surjection for $-J$, 
and by the dimensions these must both be isomorphisms.
So we need to find $|J\cap I^c|$, that is, to consider the signs on the $a_i$ in $J$. We start
by considering $a_{k_r}$.
Note that it must be $< (n+1)/2$.

\ \ \ Assume first 
that there are elements in $I_0$ which are $< a_{k_r}$, let $i$ be the largest such
element. Then in $I$, this $i$ has sign $\epsilon_{k_r}$. Then $i$ is in $J$ (since $i < (n+1)/2$)
and has sign ${\bar{\epsilon}_{k_r}}$ in $J$.  Therefore
$a_{k_r}$ has sign $\epsilon_{k_r}$ in $J$.
Now assume  there is no element in $I_0$ which is $< a_{k_r}$.
Then take $i\in I_0$ to be the smallest element, this is then $> a_{k_r}$
and has sign $\bar{\epsilon}_{k_r}$. Moreover, we must have
that $a_{k_r}$ is the smallest element of $J_0$ and then in $J$ its sign is
$\epsilon_{k_r}$,  and so it again belongs to $J\cap I^c$.

Now if $|a_{k_i}-a_{k_{i+1}}|$ is odd then they have the same
sign in $I^c$ and opposite signs in $J$, thus exactly one of them will
be in $J \cap I^c$. 
If $|a_{k_i}-a_{k_{i+1}}|$ is even then they have the opposite
sign in $I^c$ and the same signs in $J$, thus again,
exactly one of them will
be in $J \cap I^c$. 

Since $a_{k_r}^{\bar{\epsilon}_{k_r}}$ is in $J \cap I^c$ we must have
$J \cap I^c = \{\ldots, a_{{k_{r-4}}}^{\bar{\epsilon}_{{k_{r-4}}}} ,
a_{{k_{r-2}}}^{\bar{\epsilon}_{{k_{r-2}}}} , a_{k_r}^{\bar{\epsilon}_{k_r}} \}$
\ hence
$$\ext^1_{D_n}(\Delta(J),\Delta(I))=
\begin{cases}\frac{r}{2} & \mbox{if $r$ even,} \\
\frac{r+1}{2} & \mbox{if $r$ odd.} 
\end{cases}$$ 

We may now use Proposition \ref{prop:ext}~(c) and the previous
lemma to obtain:
$$\ext^1_{D_n}(\Delta(-J),\Delta(I))=
\begin{cases}\frac{r}{2} & \mbox{if $r$ even,} \\
\frac{r-1}{2} & \mbox{if $r$ odd.} \qedhere
\end{cases}$$ 
\end{proof}

We now  
construct an extension of $\Delta(I)$ by $\Delta(\pm J)$ which has no
self  extensions.
Consider the long exact sequence used to calculate the $\Ext$ group:
$$
0 
\to \Hom(\Delta(\pm J), \Delta(I))
\to \Hom(\Delta(\pm J), Q(1^+))
\to \Hom(\Delta(\pm J), \nabla(I^c))
\to \Ext^1(\Delta(\pm J), \Delta(I))
\to 0\, .
$$
Thus using the definition of the long exact sequence, we must have
that
all extensions of $\Delta(I)$ by $\Delta(\pm J)$ are constructed by
taking the pullback of an appropriate 
map from $\Delta(\pm J)$ to $\nabla(I^c)$.

Of course, in general there will be many non-split extensions of
$\Delta(I)$ by $\Delta(\pm J)$. We will construct an extension of
$\Delta(I)$ by $\Delta(\pm J)$ with no self extensions.

\begin{prop}\label{prop:homcyclic}
Let $I$ be a signed subset of $[n]^{\pm}$
such that $\Delta(I) \subset Q(1^+)$.
We let $J = \Phi(I)$ and impose the further condition that
if $i \in I_0 \setminus J_0$ then $i \ge \frac{n+1}{2}$. 
Then $\Hom_{D_n}(\Delta(J), \nabla(I^c))$ is 
cyclic as a $\Gamma$-module and
$\Hom_{D_n}(\Delta(-J), \nabla(I^c))$ is 
cyclic as a $\Gamma$-module. 
\end{prop}
\begin{proof}
We prove this for the $J$ case, the $-J$ case is similar.
We suppose that $J = \{ j_1^{\alpha_1}, j_2 ^{\alpha_2}, \ldots ,
j_{n-m}^{\alpha_{n-m}} \}$ where $\alpha_i$ is of appropriate sign.
(In fact $\alpha_i = +$ if $i$ is even and $-$ if $i$ is odd.)

In the proof of Lemma~\ref{lem:extBIJ} we showed that
$J\cap I^c= \{a_{k_e}^{\bar{\epsilon_{k_e}}},
 \ldots, a_{k_{r-2}}^{\bar{\epsilon}_{k_r-2}},
  a_{k_r}^{\bar{\epsilon}_{k_r}}\}$ where $e=1$ if $r$ is odd, and $e=2$ otherwise. 
For each $a^{\bar{\epsilon}_{k_i}}_{k_{i}} \in J\cap I^c$
there is a corresponding map $\theta_{i}$ which restricts to the
unique map (up to scalars) $\Delta(a^{\bar{\epsilon}_{k_i}}_{k_{i}}) \to
\nabla(a^{\bar{\epsilon}_{k_i}}_{k_{i}})$ with 
image $L(a^{\bar{\epsilon}_{k_i}}_{k_{i}})$
on the subquotients  $\Delta(a^{\bar{\epsilon}_{k_i}}_{k_{i}})$ 
of $\Delta(J)$ and
$\nabla(a^{\bar{\epsilon}_{k_i}}_{k_{i}})$ of $\nabla(I^c)$.
The $\theta_i$'s in fact form a basis for $\Hom_{D_n}(\Delta(J),
\nabla(I^c))$ as a $k$-vector space.

We will construct such $\theta_i$ in sufficient detail, and then show that the
map $\theta_e$ is a cyclic generator for the hom space as as $\Gamma$ module.

(1) \ \ We start with the construction.  We fix $i$ and write $a:= a_{k_i}$ whose signed version belongs to $J\cap I^c$. 
For any signed set $K$ we write
$$K_{> a} := \{ j^*\in K: j>a\}, \ \ K_{\leq a}: = \{ j^*\in K: j\leq a \}\, .
$$
Then there are short exact sequences
$$0\to \Delta(J_{> a}) \to \Delta(J)\stackrel{\pi}\to  
\Delta(J_{\leq a})\to 0
$$
$$0\to \nabla(I^c_{\leq a}) \stackrel{\kappa}\to \nabla(I^c)\to  
\nabla(I^c_{>a})\to 0\, .
$$
The signed version of $a$ belongs to both $J_{\leq a}$ and to $I^c_{\leq a}$. We will construct a  
homomorphism $\theta_i': \Delta(J_{\leq a}) \to \nabla(I^c_{\leq a})$ 
which comes as before to a non-zero map from 
$\Delta(a_{k_i}^{\bar{\epsilon}_{k_i}})$ to
$\nabla((a_{k_i}^{\bar{\epsilon}_{k_i}})$,
and then take 
$$\theta_i:= \kappa \circ \theta_i'\circ \pi.
$$
We have inclusions
$$ \Delta(I_{<a}) \ \subset \Delta(J_{\leq a})\subset T(a^*)
$$
(with $* = \bar{\epsilon}_{k_i})$. To see the first inclusion, note that 
(with $\leq$  the partial order as defined before 8.4) we have $I_{\leq a} \leq J_{\leq a}$. 

Furthermore, $T(a*)/ \Delta(I_{< a})$ is isomorphic to 
$\nabla(I^c_{\leq a})$ (which for example one can see 
working with the factor algebra $D_a$ of $D_n$, for which $T(a^*)$ is $Q(1^*)$, ie is a projective-injective module).
Therefore we take for $\theta_i'$ the composition of 
$$\Delta(J_{\leq a}) \to \Delta(J_{\leq a})/\Delta(I_{<a})  \to T(a^*)/\Delta(I_{<a}) 
\cong \nabla(I^c_{\leq a})$$
where the first map is the canonical surjection, and the second map is the inclusion.
(Each of the modules in this construction has $a^*$ as the unique highest weight with multiplicity one, and
the map is non-zero on a vector of this weight so this is a map as required).

Then the  kernel of $\theta_i$ has the exact sequence 
$$0\to \Delta(J_{>a})\to {\rm Ker}(\theta_i)\to \Delta(I_{<a})\to 0
$$
Furthermore, 
the kernel is a submodule of $\Delta(J)$ and therefore it has a simple socle. This means

(2) \  ${\rm ker}(\theta_i) = \Delta(N_i)$ where $N_i = J_{>a} \cup I_{<a}$ for
$a= a_i=  a_{k_i}^{\bar{\epsilon}_{k_i}}$.

Now let $\theta:= \theta_e$, then we claim that the kernel of $\theta_e$ is contained in the kernel
of $\theta_i$ for all $i$: 
\  We use (2) for $a=a_e$ and also for $a=a_i$. 
The sets $N_e$ and $N_i$ are both appropriately signed. So to see
that $\Delta(N_e)\subseteq \Delta(N_i)$ we only need that
the unsigned set $(N_e)_0$ is contained in the unsigned set $(N_i)_0$. 
To show this, we  only need the following. If $j\in I_0$ and $a_e > j \geq a_i$ then $j\in J_0$. But this
holds by the general hypothesis, since we have $a_e < (n+1)/2$.

(3) \ We can now prove the cyclicity, that is, to show that $\theta_i = \psi \circ \theta_e$ for some $\psi  \in \Gamma$.

Since ${\rm ker}(\theta_e)\subseteq {\rm ker}(\theta_i)$ it follows that $\theta_i$ maps the kernel of $\theta_e$ to zero.
So there is a homomorphism $\psi: \nabla(I^c) \to \nabla(I^c)$ with $\theta_i = \psi\circ \theta_e$. 
\end{proof}

\begin{cor}\label{cor:extcyclic}
The module $\Ext^1_{D_n}(\Delta(J), \Delta(I))$ 
is cyclic as a 
module for $\Gamma= \End(\nabla(I^c))$. 
\end{cor}
\begin{proof}
Using the defining sequence for $\nabla(I^c)$ we see that for 
any $D_n$-module $M$ that
$\Ext^1_{D_n}(M, \Delta(I)) \cong \uHom_{D_n}(M,\nabla(I^c))$ 
where $\uHom$ denotes the $\Hom$ space modulo homomorphisms that 
factor through a projective module. 
Thus as $\Ext^1_{D_n}(\Delta(J), \Delta(I))$ is isomorphic to 
the cyclic $\Gamma$-module  $\Hom_{D_n}(\Delta(J),\nabla(I^c))$, 
via the induced homomorphism from the long exact
sequence and this morphism is compatible with the action of $\Gamma$,
it itself must be cyclic.
\end{proof}

The following general lemma from homological algebra is well-known.
\begin{lemma}
Assume $0\to A\stackrel{j}\to B\stackrel{\pi}\to C\to 0$ is a short
exact sequence of 
finite-dimensional modules, and 
let $\xi\in {\rm Ext}^1(C,A)$ represent this sequence. Let also
$\pi^*: \Ext^1(C,A)\to \Ext^1(B,A)$ be the map induced by $\pi$. Then 
$\pi^*(\xi) = 0$. 
\end{lemma}

We now assume that 
$I$ and $J$ are (alternatingly) signed subsets of $[n]^{\pm}$
as in the beginning of this section.
I.e. $|I|=n-m$, 
$J = \Phi(I)$ and the smallest element of $J^c$ is larger that $I^c$.

Let $\xi$ be a generator of 
$\Ext^1_{D_n}(\Delta(\pm J), \Delta(I))$ as a $\Gamma$-module.
Now $\xi$ is the image of some map $\theta: 
\Delta(\pm J) \to \nabla(I^c)$ from the long exact sequence.
Thus the extension $\xi$ represents may be taken as the pullback  of
this map $\theta$.

We let $E(I,\pm J)$  be the extension $\xi$. I.e. it denotes 
the module with short exact sequence:
$$0 \to \Delta(I) \to E(I,\pm J) \to \Delta(\pm J) \to 0$$
constructed by taking the pullback of $\theta$.
We claim the following:
\begin{prop}\label{prop:E}
The module $E(I,\pm J)$ has no self-extensions.
That is, $\Ext^1_{D_n}(E(I,\pm J),E(I,\pm J)) = 0$.
\end{prop}
\begin{proof}
We prove this for the case where $E(I,J)$ is an extension of $\Delta(I)$ by
$\Delta(J$), the $\Delta(-J)$ case follows similarly.

It is clear that $\ext^1_{D_n}(E(I,J), \Delta(J))=0$, 
since both $\ext^1_{D_n}(\Delta(I),\Delta(J))$ 
and $\ext^1_{D_n}(\Delta(J), \Delta(J))$ are zero. 
So it is enough to show that $\ext^1_{D_n}(E(I,J),\Delta(I))=0$.

The map $\theta$ is chosen so that its image $\xi$ is a generator for 
$\Ext^1_{D_n}(\Delta(I), \Delta(J))$ as a module for 
$\End_{D_n}(\nabla(I^c))$. 

Apply $\Hom_{D_n}(-, \Delta(I))$ to the short exact sequence
defining $E(I,J)$, this gives
$$\ldots \to \Ext^1_{D_n}(\Delta(J), \Delta(I))\stackrel{\pi^*}
\to \Ext^1_{D_n}(E(I,J), \Delta(I))\to 
\Ext^1_{D_n}(\Delta(I), \Delta(I))=0\, .
$$
We can view these $\Ext$ groups as 
$\End_{D_n}(\nabla(I^c))=\Gamma$-modules as in the proof
of corollary \ref{cor:extcyclic}.
The map $\pi^*$ is a homomorphism of 
$\Gamma$-modules. 
It takes $\xi$ to zero. Thus as 
$\Ext_{D_n}(\Delta(J), \Delta(I))$ is cyclic as module for $\Gamma$
with generator $\xi$ 
it follows that $\pi^*=0$ and that $\Ext^1_{D_n}(E(I,J),\Delta(I))=0$.
\end{proof}

%
\section{A characterisation of $E(I,\pm J)$}\label{s:class-exts}
%

We will split  the modules $M\in \doog$ as follows.
\begin{enumerate}
\item
We say that $M\in \doog$ is a type I module if it 
is a direct sum of modules of
the form $\Delta(I_j)$ for some 
signed sets $I_j$.
\item
We say that $M\in\doog$ is a type II module if it is not a direct sum 
of modules of
the form $\Delta(I_j)$ for some 
signed sets $I_j$.
\end{enumerate}
The indecomposable 
type I modules are already classified. Potentially there may 
be indecomposable type II modules with $L(1\pm)$ occuring more
than twice in their socles. 
To classify the Richardson orbits however, we will only need
indecomposable
type II modules with at most two simples in their socles.

The extension $E(I, \pm J)$ from the previous section
is constructed as the pullback of a map $\theta$ as 
in the previous section. Thus we have the following pullback diagram:
$$
\xymatrix@R=15pt@C=15pt{
  & & 0  \ar@{->}[d] & 0 \ar@{->}[d] &
\\
  & & \ker{\tilde{\theta}} \ar@{=}[r] \ar@{->}[d] &
\ker{\theta} \ar@{->}[d] & 
\\
0 \ar@{->}[r]  &  \Delta(I) \ar@{->}[r]_{\textstyle{j}}
\ar@{=}[d]
& E(I,\pm J) \ar@{->}[r]_{\textstyle{q}}\ar@{->}[d]_{\textstyle{\tilde{\theta}}} &
\Delta(\pm J)
\ar@{->}[r] \ar@{->}[d]_{\textstyle{\theta}} & 0 
\\
0 \ar@{->}[r]  &  \Delta(I) \ar@{->}[r] 
& Q(1^+) \ar@{->}[r]^{\textstyle{\pi}} & 
\nabla(I^c)  \ar@{->}[r] & 0\smash{.}}
$$
where
$E(I,\pm J) = \{ (a,b) \in Q(1^+) \oplus \Delta(\pm J) \mid \theta(b)=\pi(a)\}$
(\cite{maclane}). Note that this implies that $\ker \theta$ is a
submodule of $E(I,\pm J)$. 

We continue with the case $\Delta(J)$, which is the one of interest for
the application, and we  use the notation as in 
\ref{prop:homcyclic} (the other case is similar).
We have seen there that 
$\ker \theta = \Delta(\tilde{J}), \ \ \tilde{J}:= N_e = J_{> a_e} \cup
I_{< a_e}.
$
We can also identify the cokernel of $\theta$ from the construction
in \ref{prop:homcyclic}, it is $\nabla(I^c_{>a}\cup J^c_{<a})$. 
It is isomorphic to the cokernel of $\tilde{\theta}$, and hence
the 
$\Im \tilde{\theta} = \Delta(\tilde{I}), \ \ 
\tilde{I} = I_{>a} \cup J_{<a}.
$
Thus $E(I, J)$ also has a short exact sequence
$$0 \to \Delta(\tilde{J}) \to E(I,J) \to \Delta(\tilde{I}) \to 0.$$
Sometimes this sequence will split and $E(I,J)$ will decompose.
There are cases, however, where this sequence does not split and
$E(I,J)$ is in fact indecomposable. This extension $E(I,J)$ is then \emph{not}
of type I. 
This is in contrast to the results of
\cite{bhrr} where all $\Delta$-filtered modules were of this type.
We now want to show that $E(I,J)$ is in  half of the cases
indecomposable, i.e. that it really is type II. \\
%
%
%
%
Recall that 
$J\cap I^c= \{a_{k_e}^{\bar{\epsilon_{k_e}}},
 \ldots, a_{k_{r-2}}^{\bar{\epsilon}_{k_r-2}},
  a_{k_r}^{\bar{\epsilon}_{k_r}}\}$ 
where $e=1$ if $r$ is odd and $e=2$ if
$r$ is even.

It is clear that 
$\tilde{I}_0 = 
\{ i \in I_0 \mid i > a_{k_e} \} \cup \{ j \in J_0 \mid j < a_{k_e} \}$
and
$\tilde{J}_0 = 
\{ j \in J_0 \mid j > a_{k_e} \} \cup \{ i \in I_0 \mid i < a_{k_e} \}$
Thus
$\tilde{J}_0= J_0 \cap I_0$ if 
$a_{k_1}^{\bar{\epsilon}} \in J \cap I^c$, i.e. if $e=1$. 
Otherwise $\tilde{J}_0 \setminus \tilde{I}_0 = \{ a_{k_1}\}$.

\begin{lemma}
$\Ext^1_{A_n}(\Delta(\tilde{I}_0), \Delta(\tilde{J}_0)) =0$ if and
only if $\tilde{I}_0 = I_0 \cup J_0$ and $\tilde{J_0} = I_0 \cap J_0$
\end{lemma}
\begin{proof}
Clearly, if $\tilde{I}_0 = I_0 \cup J_0$ and $\tilde{J} = I_0 \cap J_0$
then $\Ext^1_{D_n}(\Delta(\tilde{I}), \Delta(\tilde{J})) =0$ 
by Proposition \ref{prop:I-in-J}. 
To prove the converse, we actually calculate the $\Ext$ group
directly.
Now,
$\Hom_{A_n}(\Delta(\tilde{I}_0), \Delta(\tilde{J}_0))$ can be calculated 
in a similar fashion to $\Hom_{A_n}(\Delta(I_0), \Delta(J_0))$.
The last overlapping segment is now 
\setcounter{MaxMatrixCols}{15}
$$
\begin{matrix}
\cdots &  j_t & a_l & \cdots &\cdots   &\cdots  &\cdots  
& j_{u-1} & j_{u} &\cdots &j_{n-m}\\
\cdots &\cdots & \cdots &\cdots &  i_{s-1} & i_s  &\cdots 
&\cdots   &\cdots  & \cdots & i_{n-m} 
\end{matrix}
$$
with the same indices as in the proof of \ref{lem:extAIJ}.
Thus 
$\hom_{A_n}(\Delta(\tilde{I}_0), \Delta(\tilde{J}_0)) = |\tilde{J}_0|$
which is $n-m-r$ if $e=1$ and $n-m-r+1$ if $e=2$.

We have
$\hom_{A_n}(\Delta(\tilde{I}_0), P(1)) = |\tilde{I}_0|$
which is $n-m+r$ if $e=1$ and $n-m+r-1$ if $e=2$.

Also 
$\hom_{A_n}(\Delta(\tilde{I}_0), \nabla(\tilde{J}_0^c)) =
|\tilde{I}_0\cap \tilde{J}^c_0|$
which is $2r$ if $e=1$ and $2r-1$ if $e=2$.
This is as
$$\tilde{I}_0\cap \tilde{J}^c_0=
\{b_{k_r},b_{k_{r-1}},
 \ldots, 
  b_{k_1},
a_{k_e}, a_{k_{e+1}},
 \ldots, 
  a_{k_r}\}\, .
$$

Thus
$\ext^1_{A_n}(\Delta(\tilde{I}_0), \Delta(\tilde{J}_0))$ 
is $0$ if $e=1$ and $1$ if $e=2$.
\end{proof}

\begin{lemma}
$\Ext^1_{D_n}(\Delta(\tilde{I}), \Delta(\tilde{J})) =0$ if and
only if $\tilde{I}_0 = I_0 \cup J_0$ and $\tilde{J_0} = I_0 \cap J_0$.
\end{lemma}
\begin{proof}
We need only consider the $e=2$ case using Proposition
\ref{prop:I-in-J}.

We calculate 
$\Hom_{D_n}(\Delta(\tilde{I}), \Delta(\tilde{J}))$.
We need to determine the signs on the 
last overlapping signed segment in the previous proof. 
The sign on the last element of $\tilde{I}$ is $-$
if $|\tilde{I}|$ is even and $+$ if 
$|\tilde{I}|$ is odd.
The sign on the last element of $\tilde{J}$ is $+$
if $|\tilde{J}|$ is even and $-$ if 
$|\tilde{J}|$ is odd. 

We also know that the parities of $|\tilde{J}|$ and
$|\tilde{I}|$ are equal. Thus the signs on the last two
elements of $\tilde{I}$ and $\tilde{J}$ are always opposite.
Hence 
$\hom_{D_n}(\Delta(\tilde{I}), \Delta(\tilde{J}))$
is $\lfloor \frac{n-m-r+1}{2} \rfloor$.

Now 
$\hom_{D_n}(\Delta(\tilde{I}), P(1^-)) = 
\lfloor  \frac{n-m+r-1}{2} \rfloor$.
Also 
$\hom_{D_n}(\Delta(\tilde{I}), \nabla(\tilde{J}^c)) =
|\tilde{I}\cap \tilde{J}^c|$.
$$
\tilde{I}\cap \tilde{J}^c=
\{ b_{k_r}, \ldots, b_{k_4}, b_{k_2},
a_{k_e}, \ldots,a_{k_{r-2}}, a_{k_r}\}
$$
with appropriate signs 
by a similar argument to the one in the proof
of lemma \ref{lem:extBIJ}.
 Thus 
$\hom_{D_n}(\Delta(\tilde{I}), \nabla(\tilde{J}^c)) = r$
and
$\ext^1_{D_n}(\Delta(\tilde{I}), \Delta(\tilde{J}))
= \lfloor \frac{n-m-r+1}{2} \rfloor
- \lfloor \frac{n-m+r-1}{2} \rfloor + r = 1$.
\end{proof}

We collect information on a possible direct sum decomposition of
$E(I,\pm J)$.  
We take $E(I,J)$ in the form
$E(I,J)= \{ (x,y)\in Q(1^+)\oplus \Delta(J) \mid \pi(x)=\theta(y)\}$, the
case with $-J$ is similar.

Suppose $E(I,J)$ is a direct sum. The
socle of $E(I,J)$ is contained in $\soc(\Delta(I)\oplus \Delta(J))$, and it
follows that each summand has a simple socle $L(1\pm)$ and their 
socles are not isomorphic.
Moreover, the summands have $\Delta$-filtration, so
$E(I,J)=\Delta(L_1)\oplus \Delta(L_2)$ for some signed sets $L_1, L_2$ such that
$L_1\cup L_2 = I\cup J$.  

It follows then that $\Delta(I)$ is isomorphic to a submodule of
one of $\Delta(L_1)$ or $\Delta(L_2)$. 
Consider  the inclusion  $j: \Delta(I)\to E(I,J)$, it is 1-1 and the socle of
$\Delta(I)$ is simple. Let  $p_i: E(I,J)\to \Delta(L_i)$
be the projection onto the summand $\Delta(L_i)$
then one of $p_i\circ j$ must be non-zero on the socle
of $\Delta(I)$ and if so then $p_i\circ j$ is 1-1. 
We pick our indices so that 
$\Delta(I)$ is a submodule of $\Delta(L_1)$,
then $\Delta(L_2)$ is a submodule of 
$\Delta(J)$, as the map $q$ must be 1-1 restricted to the socle
of $\Delta(L_2)$. 

Similarly, using the other short exact sequence for $E(I,J)$, 
$\Delta(\tilde{J})$ is a submodule of one of the summands
$\Delta(L_1)$ or $\Delta(L_2)$ of $E(I,J)$. By matching the socles, we
see that $\Delta(\tilde{J}) \subset \Delta(L_2)$ and 
$\Delta(L_1) \subset \Delta(\tilde{I})$.

%

We remark that in the special case that $\theta$ is injective, then 
$E(I,J)\cong \Delta(\tilde{I})$ and hence is indecomposable. 
When $\theta$ is not injective, we have the following theorem describing 
exactly when $E(I,J)$ decomposes. 

\begin{theorem}\label{thm:Edecomp}
Assuming that $\theta$ is not injective, then
the following are equivalent.
\begin{enumerate}
\item[(i)] $E(I,J)$ is decomposable;
\item[(ii)] $E(I,J) = \Delta(L_1) \oplus \Delta(L_2)$ for some 
   signed subsets, $L_1$ and $L_2$;
\item[(iii)] $E(I,J) = \Delta(\tilde{I}) \oplus \Delta(\tilde{J})$;
\item[(iv)] $r$ is odd.
\end{enumerate}
\end{theorem}
\begin{proof}
Clearly (iii) $\Rightarrow$ (ii) and 
(i) $\Leftrightarrow$ (ii) by the discussion about the socle of $E(I,J)$.

Also (iv) $\Rightarrow$ (iii) as then 
$\Ext^1_{D_n}( \Delta(\tilde{I}), \Delta(\tilde{J}))=0$ 
and the sequence for $E(I,J)$ splits.

Next (iii) $\Rightarrow$ (iv). If (iv) is not true then
$\Ext^1_{D_n}( \Delta(\tilde{I}), \Delta(\tilde{J}))=k$ and
then 
$\Ext^1_{D_n}( E(I,J), E(I,J))=
\Ext^1_{D_n}( \Delta(\tilde{I}), \Delta(\tilde{I}))
\oplus \Ext^1_{D_n}( \Delta(\tilde{I}), \Delta(\tilde{J}))
\oplus \Ext^1_{D_n}( \Delta(\tilde{J}), \Delta(\tilde{I}))
\oplus \Ext^1_{D_n}( \Delta(\tilde{J}), \Delta(\tilde{J})) 
\ne 0$
contradicting that $E(I,J)$ has no self extensions. 

Thus it remains to prove that (ii) $\Rightarrow$ (iii).
Recall that
$\Delta(\tilde{J}) \subset \Delta(L_2) \subset \Delta(J)$ and 
$ \Delta(I) \subset \Delta(L_1) \subset \Delta(\tilde{I})$ by the
discussion
preceeding this proof. We will now show that 
$\Delta(L_2) \subset \ker \theta \cong \Delta(\tilde{J})$. 
Now  $E(I,J) \cong \Delta(L_1) \oplus \Delta(L_2)$, and
$\Ext^1_{D_n}(E(I,J),\Delta(I)) =0$ we must 
have $\Ext^1_{D_n}(\Delta(L_2), \Delta(I)) =0$.
Hence we have a short exact sequence
$$
0 \to \Hom_{D_n}( \Delta(L_2), \Delta(I))
 \to \Hom_{D_n}( \Delta(L_2), P(1^+))
 \to \Hom_{D_n}( \Delta(L_2), \nabla(I^c)) \to 0.
$$
Now $\Delta(L_2) \subset  \Delta(J)$.
As a $A_n$-module $\Delta(J)$ embeds in $\Delta(I)$
thus so does $\Delta(L_2)$. 
Hence $$\hom_{D_n}( \Delta(L_2), \Delta(I))
=\Bigl\lfloor \frac{|L_2|}{2} \Bigr\rfloor.$$
Also, 
$$\hom_{D_n}( \Delta(L_2), P(1^+))
=\Bigl\lfloor \frac{|L_2|}{2} \Bigr\rfloor.$$
Hence by dimensions
$\Hom_{D_n}( \Delta(L_2), \nabla(I^c))=0$.
Since there are no homomorphisms from $\Delta(L_2)$ to
 $\nabla(I^c)$ and $\Delta(L_2)$ is a submodule of
$\Delta(J)$, $\Delta(L_2)$ must be contained in the kernel
of $\theta$. Thus $\Delta(L_2) \subset \ker \theta \cong
\Delta(\tilde{J})$. As $\Delta(\tilde{J}) \subset \Delta(L_2)$, by
dimensions we must have $\Delta(\tilde{J}) \cong \Delta(L_2)$. 
Thus $\tilde{J}=L_2$.
Since $L_1 = I\cup J \setminus L_2= \tilde{I}$ as a multiset, 
we also have $\Delta(L_1) \cong \Delta(\tilde{I})$.
\end{proof}

\begin{ex} 
Let 
$I=\{8^+,7^-,6^+,5^-,4^+,1^-\}$ and
$J=\{8^-,5^+,4^-,3^+,2^-,1^+\}$. This is an example with $m=2$ gaps.
We have $I^c=\{3^-,2^-\}$ 
and $J^c=\{7^+,6^+\}$. If we take $\theta:\Delta(J) \to \nabla(I^c)$
which has image $\nabla(2^-)$,
then the pullback of $\theta$, the extension $E(I,J)$ has short 
exact sequence:
$$0 \to \Delta(\tilde{J}) \to E(I,J) \to \Delta(\tilde{I}) \to 0$$
where $\tilde{I}=\{8^+,7^-,6^+,5^-,4^+,2^-,1^+\}$ and
$\tilde{J}=\{8^-,5^+,4^-,3^+,1^-\}$. 
This sequence is non-split, if it did split then $L(3^+)$, which is
in the head of $\Delta(\tilde{J})$ would be in the head of $E(I,J)$. But 
$L(3^+)$ is not in the head of either $\Delta(I)$ nor $\Delta(J)$ and
so it cannot be in the head of $E(I,J)$.
Thus $E(I,J)$ is an example of a type II extension.
\end{ex}

%
\section{Constructing $M=M({\bf d})$ with 
$\Ext^1_{D_n}(M,M)=0$}\label{s:construction}
%

Let ${\bf d}=(d_1,\dots,d_n)$ be a symmetric dimension vector (i.e. 
$d_i=d_{n+1-i}$). Let 
$\sum d_i=N$. 
Then $P=P({\bf d})$ is a parabolic subgroup of 
$\SO_{N}$. 
The goal of this section is to construct a 
$\Delta$-filtered module $M$. 
In Theorem~\ref{thm:rich-module} we show that the $\Delta$-dimension vector of 
$M$ when restricted to $A_n$ is ${\bf d}$ and that $\Ext^1_{D_n}(M,M)=0$. 

We remark that in general, there are infinitely many 
(isomorphism classes of) $A_n$-modules 
having $\Delta$-support ${\bf d}$. 
In particular, the representation type 
of $\doog$ for $A_n$ is wild for $n\ge 6$, \cite[Proposition 7.2]{dr}. 


If we have an arbitrary parabolic subgroup $P=P({\bf d})$ in 
$\SO_N$ we want 
to associate to it a module $M=M({\bf d})$ 
for $D_n$ which has no self-extensions. 

In order to construct such a module $M$ without self-extension, 
we use the knowledge of Richardson orbits for parabolic 
subgroups of 
$\SO_N$. 
Constructions of Richardson elements 
for parabolic subgroups of the classical groups have been given 
in~\cite{ba} (under certain restrictions) 
and in~\cite{bg} for all parabolic subgroups of 
$\SO_{2N}$, $\SO_{2N+1}$ and $\Sp_{2N}$ (the
symplectic group). 
Our construction here is similar to the one in 
Definition 3.1 in~\cite{bg}.

The idea is to start with a symmetric dimension vector ${\bf d}$ 
and construct from it a finite sequence of dimension vectors ${\bf e}^k$ 
($k=1,2,\dots$ if $N$ is even, $k\ge 0$ if $N$ is odd) 
such that the sum of the ${\bf e}^k$ is equal 
to ${\bf d}$. 
Then to each ${\bf e}^k$ we associate a $\Delta$-filtered module 
$M({\bf e}^k)$ which has no self-extensions. 
The modules 
$M({\bf e}^k)$, $k\ge 1$, are either type I modules - in this case the direct
sum of two modules $\Delta(I)$ and $\Delta(J)$ 
for some subsets $I$ and $J$ or indecomposable 
type II modules as in the Section~\ref{s:class-exts}. 
The module $M({\bf e}^0)$, if present, is $P(1^+)$.

In a third step, we add all the $M({\bf e}^k)$ and show that the sum 
$M({\bf d}):=\oplus_{k} M({\bf e}^k)$ has the desired property, i.e. 
(i) that the 
$\Delta$-dimension vector of $M({\bf d})$ is ${\bf d}$ 
and (ii) that $M({\bf d})$ has no self-extensions. The precise 
definition is given in Definition~\ref{defn:M(d)}.

Let ${\bf d}=(d_1,\dots,d_n)$ be a symmetric dimension vector. 
The algorithm to obtain a module with $\Delta$-support ${\bf d}$ 
is the following. 
For $k=1,2,\dots, m$ (where $m=m({\bf d})\in\mathbb{N}$ is roughly half 
of the maximal entry of ${\bf d}$) 
we define dimension vectors 
${\bf f}^k$, ${\bf g}^k$, ${\bf e}^k$ using ${\bf d}^k$. 
To start we let $k=0$ and set ${\bf d}^{0}={\bf d}$. \\
\noindent
\begin{itemize}
\item[(0)] If $\sum d_i=N$ is odd, let ${\bf e}^0=(e^0_1,\dots,e^0_n)$ 
$=(1,1,\dots,1)$ and replace ${\bf d}^0$ by ${\bf d}^0-{\bf e}^0$. 
\item[(1)] Assume that ${\bf d}^0$, ${\bf d}^1$, $\dots, {\bf d}^k$ 
are defined. 
\item[(2)]
Define ${\bf f}^{k+1}=((f^{k+1})_1,\dots,(f^{k+1})_n)$ and 
${\bf g}^{k+1}=((g^{k+1})_1,\dots,(g^{k+1})_n)$ 
by setting 

$(f^{k+1})_i:=\left\{ 
\begin{array}{ll} 
1 & \mbox{if } ({\bf{d}}^k)_i\ge2 \\
1 & \mbox{if } ({\bf{d}}^k)_i=1 \mbox{ and } i<\frac{n+1}{2} \\
0 & \mbox{else;}
\end{array}\right.$ 

$(g^{k+1})_i:=\left\{
\begin{array}{ll} 
1 & \mbox{if } ({\bf{d}}^k)_i\ge2 \\
1 & \mbox{if } ({\bf{d}}^k)_i=1 \mbox{ and } i>\frac{n+1}{2} \\
0 & \mbox{else.}
\end{array}\right.$

\item[(3)]
Let ${\bf e}^{k+1}=((e^{k+1})_1,\dots,(e^{k+1})_n)$ where 
$(e^{k+1})_i:=(f^{k+1})_i+(g^{k+1})_i$. And then 
set ${\bf d}^{k+1}:={\bf d}^k-{\bf e}^{k+1}$. 
If ${\bf d}^{k+1}= (0,0,\dots,0)$ we are done. Otherwise, 
we continue this procedure by going back to step (1).  
\end{itemize}

This gives a sequence of dimension vectors ${\bf e}^k$, $k\ge 1$, with entries at
most $2$ 
and such that $\sum {\bf e}^k={\bf d}$ (coordinate-wise sum). If 
$N$ is odd, there is in addition 
a dimension vector ${\bf e}^0$ consisting only of $1$'s. 
Also, we obtain two decreasing sequences $I_0^k$, $J_0^k$ of subsets of $[n]$ 
as follows: 

\noindent
For $k\ge 1$ define $I_0^k$ and $J_0^k$ 
to be the support of ${\bf f}^k$ 
and of ${\bf g}^k$ respectively i.e. 
$I_0^k:=\{i\mid (f^k)_i\neq 0\}$ and let 
$J_0^k:=\{i\mid (g^k)_i\neq 0\}$, always in decreasing order. 
I.e. if the first entry of ${\bf f}^k$ (of ${\bf g}^k$) is non-zero, $I_0^k$ 
($J_0^k$, respectively) contains $n$. If the second entry 
of ${\bf f}^k$ is non-zero, $I_0^k$ contains $n-1$, etc.  
For odd $N$ set $I_0^0=[n]=\{n,n-1,\dots,1\}$. 

\begin{remark}\label{rem:I_k-inclusions}
Note that we have $I_0^k\supset I_0^{k+1}$  
and $J_0^k\supset J_0^{k+1}$ for $k=1,2,\dots$, and 
if $N$ is odd, $I_0^0\supset I_0^1$, $I_0^0\supset J_0^1$. 
Furthermore, the support of ${\bf d}$ (i.e. the set of indices of the nonzero 
entries) is equal to $I_0^1\cup J_0^1$ if $N$ is even and 
equal to $I_0^0$ if $N$ is odd.
\end{remark}

\begin{lemma}\label{lem:subsets}
We have $I_0^k\subset J_0^{k-1}$  and
$J_0^k\subset I_0^{k-1}$ for $k\ge 2$.  
\end{lemma}
\begin{proof}
Consider the vectors ${\bf f}^k$ and ${\bf g}^k$. By definition
$({\bf f}^k)_i = ({\bf g}^k)_i$ unless $(d^k)_i = 1$.
Thus  $I_0^k \cap J_0^k = \{ i \mid ({\bf f}^k)_i = ({\bf g}^k)_i
=1\}$.  Also, by construction, $({\bf f}^k)_i \ne 0$
implies that $({\bf f}^{k-1})_i \ne 0$. (This is why
$I_0^k \subset I_0^{k-1}$ and 
$J_0^k \subset J_0^{k-1}$.) 
Thus if $i \in I_0^k \cap J_0^k$ then $i$ is in both
$I_0^{k-1}$ and $J_0^{k-1}$. 
 
If $i \in I_0^k \setminus J_0^k$ then $({\bf f}^k)_i =1$
and $({\bf g}^k)_i =0$. Thus $({\bf d}^k)_i =1$ and 
$({\bf d}^{k-1})_i =3$. Hence 
$({\bf f}^{k-1})_i = ({\bf g}^{k-1})_i$ and 
$i$ is in both 
$I_0^{k-1}$ and $J_0^{k-1}$. 

Hence $I_0^k \subset J_0^{k-1}$. Similarly $J_0^k \subset I_0^{k-1}$.
\end{proof}

For each of the $I_0^k$, $J_0^k$, $k\ge 1$ 
we let $I^k$ and $J^k$ be signed versions 
such that $\Delta(I^k)\subset Q(1^+)$ and $\Delta(J^k)\subset Q(1^-)$. 
In particular, the largest entry of $I^k$ has a positive sign 
and the largest entry of $J^k$ has negative sign. 
If $N$ is odd, let $I^0$ be the signed subset such that 
$\Delta(I^0)=P(1^+)$. 

\begin{definition}\label{defn:M(d)}
Let ${\bf d}$ be a symmetric dimension vector and ${\bf e}^k$ 
the vectors as defined above. The module $M({\bf d})$ is then 
defined as follows: 
In the case where $N$ is odd, we set $M({\bf e}^0):=Q(1^+)$. 
For $k\ge 1$: \\
If $I_0^k=J_0^k$ we define $M({\bf e}_k):=\Delta(I^k)\oplus\Delta(J^k)$. 
Otherwise, $M({\bf e}^k)$ is defined to be the unique module obtained 
from $\Ext^1_{D_n}(\Delta(I^k),\Delta(J^k))$ with no self-extensions 
as in Proposition~\ref{prop:E}, Section~\ref{s:gaps}. 
Then the module $M({\bf d})$ is set to be the sum of all $M({\bf e}^k)$: 
\begin{eqnarray*}
M({\bf d}) & := & \bigoplus_{k\ge 1} M({\bf e}^k) \quad \text{if $N$ is even}\\
    & := & \bigoplus_{k\ge 0} M({\bf e}^k) \quad \text{if $N$ is odd}
\end{eqnarray*}
\end{definition}

\begin{ex}\label{ex:construction}
We illustrate the construction with the examples ${\bf d}=(1,3,5,4,5,3,1)$ 
and ${\bf d}=(1,3,5,3,5,3,1)$. 

\begin{itemize}
\item[(a)]  ${\bf d}=(1,3,5,4,5,3,1)$. \\
In a first step, 

${\bf f}^1=(1,1,1,1,1,1,0)$ and ${\bf g}^1=(0,1,1,1,1,1,1)$, 

${\bf f}^2=(0,1,1,1,1,0,0)$ and ${\bf g}^2=(0,0,1,1,1,1,0)$, 

${\bf f}^3=(0,0,1,0,0,0,0)$ and ${\bf g}^3=(0,0,0,0,1,0,0)$.

\noindent
From this we get 

${\bf e}^1=(1,2,2,2,2,2,1)$,

${\bf e}^2=(0,1,2,2,2,1,0)$,

${\bf e}^3=(0,0,1,0,1,0,0)$. 

\noindent
The signed subsets are 

$I^1=\{7^+,6^-,5^+,4^-,3^+,2^-\}$ and $J^1=\{6^-,5^+,4^-,3^+,2^-,1^+\}$, 

$I^2=\{6^+,5^-,4^+,3^-\}$ and $J^2=\{5^-,4^+,3^-,2^+\}$, 

$I^3=\{5^+\}$ and $J^3=\{3^-\}$. 

\noindent
From this, $M({\bf d})=M({\bf e}^1) \oplus M({\bf e}^2) \oplus M({\bf e}^3)$. 
All the $M({\bf e}^k)$ are obtained as extensions. 
\item[(b)] ${\bf d}=(1,3,5,3,5,3,1)$. \\
In a first step, 

${\bf e}^0=(1,1,1,1,1,1,1)$,

${\bf f}^1=(0,1,1,1,1,1,0)$ and ${\bf g}^1=(0,1,1,1,1,1,0)$, 

${\bf f}^2=(0,0,1,0,1,0,0)$ and ${\bf g}^2=(0,0,1,0,1,0,0)$.  

So this gives 

${\bf e}^1=(0,2,2,2,2,2,0)$,

${\bf e}^2=(0,0,2,0,2,0,0)$,

The signed subsets are 

$I^0=\{7^+,6^-,5^+,4^-,3^+,2^-,1^+\}$, 

$I^1=\{6^+,5^-,4^+,3^-,2^+\}$ and $J^1=\{6^-,5^+,4^-,3^+,2^-\}$, 

$I^2=\{5^+,3^-\}$ and $J^2=\{5^-,3^+\}$. 

\noindent
From this, $M({\bf d})=M({\bf e}^0) \oplus M({\bf e}^1) \oplus M({\bf e}^2)$ 
with $M({\bf e}^0)=P(1^+)$, $M({\bf e}^1)=\Delta(I^1)\oplus\Delta(J^1)$ 
and $M({\bf e}^2)=\Delta(I^2)\oplus\Delta(J^2)$. 
\end{itemize}
\end{ex}

\begin{lemma}\label{lm:e_k-exts}
We have $\Ext^1_{D_n}(M({\bf e}^k),M({\bf e}^k))=0$ for all $k$. 
\end{lemma}
\begin{proof}
If $I_0^k=J_0^k$ then $M({\bf e}^k)$ is a sum of two type I 
modules with identical support: 
$M({\bf e}^k)=\Delta(I^k)\oplus\Delta(J^k)$ with 
$I_0^k=J_0^k$. By Proposition~\ref{prop:I-in-J} 
this has no self extensions. 
In the other case the claim follows from Proposition~\ref{prop:E}. 
\end{proof}

\begin{theorem}\label{thm:rich-module}
Let $M({\bf d})$ be the module as constructed above. 
Then we have 

(i) $M({\bf d})$ has no self-extensions;

(ii) The $\Delta_{A_n}$-support of $M({\bf d})\downarrow_{A_n}$ 
is ${\bf d}$. 
\end{theorem}

\begin{proof}
(i) 
Using Lemma \ref{lem:subsets} and applying Proposition
\ref{prop:I-in-J} we see that 
$\Ext^1_{D_n}(\Delta(I^k), \Delta(J^l)) =0 $
and $\Ext^1_{D_n}(\Delta(J^k), \Delta(I^l)) =0 $ for all $k \ne l$.
Thus using the short exact sequences for $M({\bf{d}}^k)$ and
$M({\bf{d}}^l)$ we see that
$\Ext^1_{D_n}(M({\bf{d}}^k), M({\bf{d}}^l))=0$ for all $k \ne l $.
Since, by construction 
$\Ext^1_{D_n}(M({\bf{d}}^k), M({\bf{d}}^k))=0$, it follows that
$\Ext^1_{D_n}(M({\bf{d}}), M({\bf{d}}))=0$.

(ii) follows from the construction. 
\end{proof}

\begin{remark}
We observe that in the case where for some $k\ge 1$ the subsets 
$I_0^k$ and $J_0^k$ are different then this procedure involves 
a choice of signs: the signs of $I^k$ and $J^k$ are given by the 
requirement that $\Delta(I^k)$ is a submodule of $Q(1^+)$ and 
that $\Delta(J^k)\subset Q(1^-)$. We could equally well 
take $-I^k$ and $-J^k$ instead. 

This ambiguity arises unless all $d_i$ are even. 
If all the $d_i$ are even, the module $M({\bf d})$ 
is induced by the module $\Delta({\bf d})$ from Section 2 of~\cite{bhrr}, 
$M({\bf d})=\Delta({\bf d})\otimes_{A_n}D_n$. 
\end{remark}

\begin{lemma}\label{lm:signs}
Let ${\bf d}=(d_1,\dots,d_n)$ be a symmetric dimension vector. Let 
$s_{odd}$ be the number of different odd entries of ${\bf d}$ and $s_{even}$ 
the number of different even entries of ${\bf d}$. 

If $N$ is even then 
there are $2^{s_{odd}}$ different $D_n$-modules without 
self-extensions such that 
their restriction to $A_n$ has $\Delta$-support ${\bf d}$. 

If $N$ is odd then there are $2^{1+s_{even}}$ different $D_n$-modules without 
self-extensions such that 
their restriction to $A_n$ has $\Delta$-support ${\bf d}$. 
\end{lemma}

\begin{proof}
We will proof the case of even $N$. For the odd case, observe that 
$M({\bf e}^0)=Q(1^+)$ and $Q(1^-)$ both have the same $\Delta$-support 
$[n]$ when restricted to $A_n$. Thus we are left to understand 
$\oplus_k M({\bf e}^k)$ for $k>0$. This is equivalent to consider 
the module $M({\bf d}-{\bf e}^0)$ and $N-n$. This number is even, 
since $n$ is odd for odd orthogonal groups. The number of 
odd entries of ${\bf d}-{\bf e}^0$ is just the number of even 
entries of ${\bf d}$. 
So the case of odd $N$ reduces to the even case after 
subtracting $1$ from all the $d_i$. \\

Let now $N$ be even. 
Set $\widetilde{m}$ to be the smallest entry of ${\bf d}$. If it is even, let 
$m:=\widetilde{m}/2$. Then the algorithm obtains 
as ${\bf e}^1,\dots,{\bf e}^m$ dimension vectors consisting only of $2$'s 
and the corresponding modules $M({\bf e}^k)$ are 
$\Delta(I^k)\oplus \Delta(J^k)=\Delta(I^k)\oplus \Delta(-I^k)$, $1\le k\le m$. 
So the first $m$ summands of $M({\bf d})$ are uniquely determined 
and we can ignore them: W.l.o.g. let the minimal entry $\widetilde{m}$ 
of ${\bf d}$ be odd. 
Then $I^1_0\neq J^1_0$, i.e. $\Delta(I^1)\neq \Delta(-J^1)$. In particular, if 
we set $\widetilde{M}({\bf e}^1)$ to be the the unique module obtained 
from $\Ext^1_{D_n}(\Delta(-I^k),\Delta(-J^k))$ with no self-extensions 
as in Proposition~\ref{prop:E}, Section~\ref{s:gaps}, then 
the restrictions to $A_n$ of $M({\bf e}^1)$ and of $\widetilde{M}({\bf e}^1)$ 
are identical, but $M({\bf e}^1)\neq \widetilde{M}({\bf e}^1)$. 

Now by step (3) of the algorithm, the remaining dimension vector 
is ${\bf d}^2$. Let $\widetilde{m}_2$ be its minimal entry. If it is even, 
the algorithm produces $\widetilde{m}_2/2$ vectors ${\bf e}^k$ whose entries 
are only $0$s and $2$s. The corresponding $I^k$ are equal to $-J^k$ 
and thus again we have 
$M({\bf e}^k)=\Delta(I^k)\oplus \Delta(J^k)=\Delta(I^k)\oplus \Delta(-I^k)$. 
So the only interesting thing happens when $\widetilde{m}_2$ is odd. 
In that case, the same reasoning as above shows that there are 
two $D_n$-modules with identical restriction to $A_n$. 

Therefore, each different odd entry of ${\bf d}$ produces a pair of 
$D_n$-modules with no self-extensions and with identical 
$\Delta$-support when restricted to $A_n$. 
\end{proof}

As an illustration of this: in Example~\ref{ex:construction}, part (a) 
for each of the summands $M({\bf e}^k)$ there is another module with 
same $\Delta$-support when restricted to $A_n$. In part (b), only 
for $M({\bf e}^0)$  
there is another module with the same $\Delta$-support when 
restricted to $A_n$, namely $Q(1^-)$.

%
\begin{appendix}
\section{$R_n$, $S_n$ and their Auslander algebras}
%

We fix a field $k$ of characteristic $\neq 2$. 
Recall that we have defined $R_n:=k[T]/T^n$.
The algebra $R_n$ has  precisely $n$ indecomposable modules,
of dimensions $1,2, \ldots, n$. Following~\cite{bhrr} we write
$M(i)$ for the indecomposable module of dimension $n-i+1$. 

We work with right modules, and we write maps to the left.
The Auslander algebra of $R_n$
is then by definition 
the algebra $A_n:=\End(M)$ where $M:=\bigoplus_{i=1}^n M(i)$.
In~\cite{bhrr}, a presentation by quiver and relations is given. 
Since we will need it in the proof of Proposition~\ref{prop:Dn-skew}, 
we give explicit generators for $A_n$.

For each $i$ with $1\leq i\leq n$ take a basis of $M(i)$,
$$\{ b_j^{(i)}: 1\leq j\leq n-i+1\}
$$
such that $(b_j^{(i)})T= b_{j+1}^{(i)}$ 
(with the convention that $b_{n-i+2}^{(i)} = 0$)
Then $M$ has basis  ${\mathcal B}$, the union of all these bases.
The 
algebra $A_n$ is generated, as an algebra, by inclusion 
maps $\alpha_{a-1}: M(a)\to M(a-1)$
together with maps $\beta_{a+1}: M(a)\to M(a+1)$, which are surjections.
We fix such maps explicitly, as
$$\alpha_{a-1}(b_i^{(a)}) := b_{i+1}^{(a-1)}, \ \ \beta_{a+1}(b_i^{(a)}) := b_i^{(a+1)}
$$
(for $1\leq i\leq n-a+1$ and with the obvious conventions).

This gives directly the quiver and relations. 
%
%
$A_n$ 
is  given by a quiver $Q_n$ with $n$ vertices
$\{1,2,\dots,n\}$
and $2n-2$ arrows between them,
$\alpha_i:i\to i+1$ for $i=1,\dots,n-1$ and
$\beta_i:i\to i-1$ for $i=2,\dots,n$,
subject to the relations
$\beta_i\alpha_{i-1}=\alpha_i\beta_{i+1}$ for $1<i<n$ and
$\beta_n\alpha_{n-1}=0$.


%

%
%

The algebra $S_n$ also is of finite type (for details, see below). 
We  let $D_n$ be its Auslander
algebra, this is of main interest to us. We will first state
its presentation by quiver and relations, and below we will show that
it is isomorphic to a skew group ring of $A_n$ with a cyclic group
of order 2 (which then will also prove the presentation).

\noindent We start by defining a quiver of cylindrical shape.
\begin{definition}\label{defn:Gamma_n}
For $n\geq 2$, let  $\Gamma_n$ be the quiver with vertices
$\{1^+,1^-,2^+, 2^-,\dots,n^+, n^-\}$ and
arrows $\alpha_{i^{\pm}}$ going from
$i^{\pm}$ to $(i+1)^{\mp}$,
and $\beta_{i^{\pm}}$ going from
$i^{\pm}$ to $(i-1)^{\pm}$, that is 
\[
\left\{
\begin{array}{lll}
\alpha_{i^+}:\ i^+\to(i+1)^- ,& \alpha_{i^-}:\ i^-\to(i+1)^+,
 & 1\le i<n\\
\beta_{i^+}:\ i^+\to(i-1)^+ ,&
\beta_{i^-}:\ i^-\to(i-1)^-, & 1< i\le n
\end{array}
\right\} \, .
\]
\end{definition}

\noindent 

Then the Auslander algebra of $S_n$ is given by the quiver $\Gamma_n$
subject to the relations
\[
\begin{array}{ll}
\beta_{i^+}\alpha_{(i-1)^+}=\alpha_{i^+}\beta_{(i+1)^-}, & 1<i<n \\
\beta_{i^-}\alpha_{(i-1)^-} =\alpha_{i^-}\beta_{(i+1)^+}, & 1<i<n \\
\beta_{n^-}\alpha_{(n-1)^-}=\beta_{n^+}\alpha_{(n-1)^+}=0.
\end{array}
\]
One way to see that this is a presentation of $D_n$ is via Auslander-Reiten
theory.
If we let $\tau$ be the map sending
$i^{\pm}$ to $i^{\mp}$ for $i=2,\dots,n$ then
$(\Gamma_n,\tau)$ is a  translation quiver with
projective vertices $\{1^+,1^-\}$. It is precisely the Auslander-Reiten 
quiver of the algebra $S_n$, and the relations are the
`mesh relations'.

%
\subsection{The indecomposable $S_n$-modules}
%

We have defined 
$S_n= R_n\langle g\rangle$, the skew group algebra,
where $g(T)=-T$ (see the introduction). 
We identify the subalgebra $R_n\otimes 1$  of $S_n\cong R_n\otimes k\langle g\rangle$ with $R_n$
and the subalgebra $1\otimes k\langle g\rangle$ with $k\langle g\rangle$.
The algebra $S_n$ has orthogonal idempotents
$e_0, e_1$ with $1=e_0 + e_1$, where 
$$e_0 = \frac{1}{2}(1+g), \ \ e_1 := \frac{1}{2}(1-g).
$$
Then $S_n =  e_0S_n \oplus e_1S_n$ as $S_n$-modules.
A basis for $e_iS_n$ is given by (the cosets of)
$$e_i, \ e_iT, \ e_iT^2, \ldots, e_iT^{n-1}\, .
$$
In particular the $e_iS_n$ are uniserial of length $n$, and are indecomposable.
One checks that $e_0T = Te_1$ and $e_1T= Te_0$, which implies
that the composition factors of $e_iS_n$ alternate.
We write $L^+$ for the simple top of $e_0S_n$, and $L^-$ for 
the simple top of $e_1S_n$. Then
$g$ has eigenvalue $1$ on $L^+$ and eigenvalue $-1$ on $L^-$.

This shows that
the quiver of $S_n$ has two vertices
which we denote by $+$
and $-$, and  two arrows,
one from $+$ to $-$ and one
from $-$ to $+$. The relations are 
that any path of length $\geq n$ is zero.
$$
{\small
\xymatrix{
 + \ar@/^/[r] & - \ar@/^/[l]
 }}
$$
$S_n$ is also a
self-injective Nakayama algebra.
In particular it has finite type, there are 
$2n$ indecomposable  modules (up to isomorphism) 
and each indecomposable
module is uniserial. 

Since $g$ has order 2 and the field has characteristic not equal to 
$2$, every $S_n$-module $X$ is relative $R_n$-projective, that is,
the multiplication map
$X\otimes_{R_n}S_n \to X$
splits. 

It is easy to construct the indecomposable $S_n$-modules by inducing
from $R_n$.
Using the explicit basis for the $R_n$-module  $M(i)$
given above, 
$M(i)\otimes _{R_n}S_n$ has basis 
$\{ b_j\otimes e_0, \ b_j\otimes e_1: 1\leq j\leq n-i+1\}$
(omitting $(-)^{(i)}$ since we fix $i$ for the moment).
An easy check gives
$$b_j\otimes e_0T = b_jT\otimes e_1 = b_{j+1}\otimes e_1,
$$
and similarly $b_j\otimes e_1T = b_{j+1}\otimes e_0$.
Since $T$ generates the algebra $R_n$, this shows that the induced module
is the direct sum of $M(i^+)$ with $M(i^-)$ where
$M(i^+)$ has basis
$$\{ b_1\otimes e_0, b_2\otimes e_1, b_3\otimes e_0, b_4\otimes e_1, \ldots \}
$$
and similarly for $M(i^-)$.
The top of $M(i^+)$ is $L(i^+)$, similarly for $M(i^-)$. 
This gives $2n$ uniserial modules for $S_n$ which clearly are pairwise
non-isomorphic. Hence this is a full set of the indecomposable 
$S_n$-modules. It follows that
$$D_n = {\rm End}_{S_n} (M\otimes_{R_n}S_n).
$$
With the explicit description of the modules $M(i^{\pm})$ it is
easy to write down maps which lead to the relations stated in 
\ref{defn:Gamma_n}. By using standard methods, it is not difficult
to prove the following, and we omit details.

\begin{prop}\label{prop:Dn-skew}
The algebra $D_n$ is isomorphic to a skew group
algebra $A_n\langle \bar{g}\rangle$ where $\bar{g}$ has order 2.
\end{prop}




%

%


The Auslander-Reiten quiver $\Gamma_n$ of $S_n$ (hence the quiver
of $D_n$) is isomorphic to $\mathbb{Z}A_n/\tau^2$, the irreducible maps are precisely
the inclusions of radicals, and taking the socle quotients. One checks that
with appropriate labelling, 
it gives precisely the quiver of the algebra $D_n$

%
%
The Auslander-Reiten quiver of
$S_n$ looks like a cylinder (as does the AR quiver of any 
self-injective 
Nakayama algebra). In general, for $k$ algebraically closed
not of characteristic $2$, the quiver for the Auslander
algebra of an algebra of finite type is the Auslander-Reiten
quiver of this algebra, and the relations
are the `mesh relations' (see \cite[p.232]{ars}). 
This gives us: 

\begin{prop}\label{prop:AR-quiver}
$\Gamma_n$ is the Auslander-Reiten quiver of $S_n$.
\end{prop}

\end{appendix}

%

\end{document}